\documentclass[11pt]{article}
\usepackage{ulem}
\usepackage{graphicx,amsmath,amsfonts,amssymb,color,bbm}
\usepackage{color}

\usepackage{epsfig}

 \newcommand{\pend}{\hfill \thicklines \framebox(5.5,5.5)[l]{}}
 \newenvironment{proof}{\noindent {\sc  Proof.} \rm}{\pend}

\numberwithin{equation}{section}

 \newtheorem{theorem}{Theorem}[section]
 \newtheorem{lemma}{Lemma}[section]

 \newtheorem{remark}{Remark}[section]

 \newtheorem{definition}{Definition}[section]

\begin{document}
 \pagenumbering{arabic} \thispagestyle{empty}
\setcounter{page}{1}

\begin{center}


\begin{center}{\Large{\bf Refined Tail Asymptotic Properties for the $M^X/G/1$ Retrial Queue}}

\

 { Bin Liu \footnotemark[1]$^{,a}$,\quad Jie Min $^a$ \quad Yiqiang Q. Zhao $^b$
 \\ {\small a. School of Mathematics \& Physics, Anhui Jianzhu University,
Hefei 230601, P.R.China}\\
{\small b. School of Mathematics and Statistics,
Carleton University, Ottawa, ON, Canada K1S 5B6}}
\end{center}

September 2019
\date{\today}

\end{center}
\footnotetext[1]{Corresponding author: Bin Liu, E-mail address:
bliu@amt.ac.cn}

\begin{abstract}
In the literature, retrial queues with batch arrivals and heavy service times have been studied and the so-called equivalence theorem has been established under the condition that the service time is heavier than the batch size. The equivalence theorem provides the distribution (or tail) equivalence between the total number of customers in the system for the retrial queue and the total number of customers in the corresponding standard (non-retrial) queue. In this paper, under the assumption of regularly varying tails, we eliminate this condition by allowing that the service time can be either heavier or lighter than the batch size. The main contribution made in this paper is an asymptotic characterization of the difference between two tail probabilities: the probability of the total number of customers in the system for the $M^X/G/1$ retrial queue and the probability of the total number of customers in the corresponding standard (non-retrial) queue. The equivalence theorem by allowing a heavier batch size is another contribution in this paper.

\medskip

\end{abstract}


\section{Introduction}\label{intro}

Studies of tail asymptotic properties, expressed in terms of simple functions, often lead to approximations, error bounds for system performance, and computational algorithms, besides their own interest. These studies become more important when closed-form or explicit solutions are not expected. On the one hand, except for a very limited number of basic queueing models, it is not in general expected to have a simple closed-form or explicit solution for the stationary queue length or waiting time distribution when it exists, but on the other hand expressions or presentations in many cases do exist for the distribution in terms of transformations, say the generating function (GF) for the stationary queue length distribution or the Laplace-Stieltjes transform (LST) of the stationary waiting time distribution. These expressions or presentations (for the transformation of the distribution) mathematically contain complete amount of information about the distribution, but they cannot be theoretically inverted to simple or closed formulas or expressions for the distribution. Many retrial queues are such examples, for which we do not expect, in general, closed-form or explicit solutions for the stationary distribution of the queue-length process or the waiting time under a stability condition. However, expressions for the transform of the distribution are available, in terms of which tail asymptotic analysis might prevail.

It is our focus in this paper to carry out an asymptotic analysis for a type of retrial queues with batch arrivals, referred to as $M^X/G/1$ retrial queues. Studies on retrial queues are extensive during the past 30 years or so. Research outcomes and progress have been reported in more than 100 publications due to the importance of retrial queues in applications, as such in the areas of call centres, computer and telecommunication networks among many others. Earlier surveys or books include Yang and Templeton~\cite{Yang-Templeton:1987}, Falin~\cite{Falin:1990}, Kulkarni and Liang~\cite{Kulkarni-Liang:1997}, Falin and Templeton~\cite{Falin-Templeton:1997}, Artalejo~\cite{Artalejo:1999a, Artalejo:1999b}, and Artajelo and G\'{o}mez-Corral~\cite{Artalejo-GomezCorral:2008}, Artalejo~\cite{Artalejo:2010} and Kim and Kim~\cite{Kim-Kim:2016} are among the recent ones.
Studies on tail behaviour can be classified into two categories: light-tail and heavy tail.
For light-tailed behaviour, references include Kim, Kim and Ko~\cite{Kim-Kim-Ko:2007}, Liu and Zhao~\cite{Liu-Zhao:2010}, Kim, Kim and Kim~\cite{Kim-Kim-Kim:2010a}, Liu, Wang and Zhao~\cite{Liu-Wang-Zhao:2012}, Kim, Kim and Kim~\cite{Kim-Kim-Kim:2012}, Kim and Kim~\cite{Kim-Kim:2012}, Artalejo and Phung-Duc~\cite{Artalejo-PhungDuc:2013}, Kim~\cite{Kim:2015},
while for heavy-tailed behaviour, readers may refer to Shang, Liu and Li~\cite{Shang-Liu-Li:2006}, Kim, Kim and Kim~\cite{Kim-Kim-Kim:2010c},
Yamamuro~\cite{Yamamuro:2012}, Liu, Wang and Zhao~\cite{Liu-Wang-Zhao:2014}, and Masuyama~\cite{Masuyama:2014}.
In addition, there are many references in the literature for asymptotic analysis for the corresponding non-retrial queues, e.g., Asmussen, Kl\"{u}ppelberg and Sigman~\cite{Asmussen-Klupperlberg-Sigman:1999}, and Boxma and Denisov~\cite{Boxma-Denisov:2007}. For more references, we would like to mention two excellent surveys: Borst et al.~\cite{Borst-Boxma-Nunez-Queija-Zwart:2003}, and Boxma and Zwart~\cite{Boxma-Zwart:2007}.

Closely related to the model of our interest in this paper are references
\cite{Shang-Liu-Li:2006}, in which it was proved that if the number of customers in the standard $M/G/1$ queue has a subexponential distribution, then the number of customers in the corresponding $M/G/1$ retrial queue has the same tail asymptotic behaviour (referred to as the equivalence theorem);
\cite{Yamamuro:2012}, in which the same result as in \cite{Shang-Liu-Li:2006} was proved for the batch arrival $M^X/G/1$ retrial queue under the condition that the batch size has a finite exponential moment; and \cite{Masuyama:2014}, in which the main result in \cite{Yamamuro:2012} was extended to a $BMAP/G/1$ retrial queue.

It has been noticed that in the literature, for a retrial queue with batch arrivals and general service times, the impact of the arrival batch on the tail equivalence property has not been sufficiently addressed. For example, in \cite{Yamamuro:2012} for the $M^X/G/1$ retrial queue, it is assumed that the arrival batch has a finite exponential moment; or in \cite{Masuyama:2014} for the $BMAP/G/1$ retrial queue, the light-tailed condition was relaxed to possibly moderately heavy-tailed batches (see Asmussen, Kl\"{u}pperlberg and Sigman~\cite{Asmussen-Klupperlberg-Sigman:1999} for a definition, i.e., the batch size has a tail not heavier than $e^{-\sqrt{x}}$). The common feature in both situations is the fact that compared to the batch size, the tail of the service time is heavier. To the best of our knowledge, in the literature, there is no report on the tail equivalence between a standard batch arrival queue and its corresponding retrial queue if the the arrival batch size is heavier than or equivalent to the service time.

For approving the equivalence theorem, it is usually to establish a stochastic decomposition first. This decomposition writes the total number of customers in the system for the retrial queue as the sum of the total number of customers in the system for the corresponding (non-retrial) queue and another independent random variable. The equivalence theorem is to prove that the total number of customers in the system for the retrial queue and the total number of customers in the system for the corresponding non-retrial queue have the same type of tail asymptotic behaviour. That has been done in the literature for the $M/G/1$ case, and extended to the $M^X/G/1$ and $BMAP/G/1$ cases under the assumption that the batch size is lighter than the service time. In terms of the decomposition, it implies that the other variable is simply dominated by the total number of customers in the system of the standard (non-retrial) model. Therefore, no detailed analysis for the other variable is needed for establishing the equivalence.

In this paper, we consider the $M^X/G/1$ retrial queue, the same model studied in \cite{Yamamuro:2012}. The equivalence theorem is now proved for the case in which the batch size has regularly varying tail, so it is heavier than the moderately heavy tail and without the assumption that the service time is heavier than the batch size. Another more interesting result (our main contribution in this paper) is an asymptotic characterization of the difference between two tail probabilities: the probability of the total number of customers in the system for the $M^X/G/1$ retrial queue and the probability of the total number of customers in the corresponding standard (non-retrial) queue. The difference between the total number $L_{\mu}$ of customers in the system for the $M^X/G/1$ retrial queue and the total number $L_{\infty}$ of customers in the corresponding standard (non-retrial) queue is the negligent (dominated) variable when establishing the equivalence theorem and therefore the asymptotic behaviour in the tail probability of this difference has not been studied in the literature. The main results of this paper are stated in Theorem~\ref{Theorem refined}.

The rest of the paper is organized as follows:
in Section~\ref{sec:2}, we describe the $M^X/G/1$ retrial queue model and rewrite the GF (a literature result) for $D^{(0)}$ (we indeed have $L_{\mu}=L_{\infty}+D^{(0)}$ in terms of the stochastic decomposition; in Section~\ref{decomposition},
a further decomposition, together with its analysis, of each component in the decomposition in Section~\ref{sec:2} is provided;
in Section~\ref{sec:4}, asymptotic analysis on the components in the decompositions given in Section~\ref{decomposition} is carried out; we complete the proof to our key result (the tail asymptotic behaviour of $D^{(0)}$) in Section~\ref{sec:5}; the refined tail equivalence theorem (main) for the total number of customers is proved in Section~\ref{sec:6}; the asymptotic tail behavior for $D^{(1)}$ is provided in Section~\ref{sec:7}; and discussions on the key condition of Lemma~\ref{lemma-F1F2} and two examples are provided in the final section. Appendix~\ref{app-A} contains some of the literature results, together with our verified preliminary results, needed for proving our main theorem; and Appendix~\ref{app-B} provides the proofs for the $\Delta$-analyticity, required for the discussions of two types of examples in Section~\ref{sec:8}.

\section{Priliminaries}
\label{sec:2}


In this paper, we consider the $M^X/G/1$ retrial queue (the same model considered in \cite{Yamamuro:2012}), in which the primary customers arrive in batches, the successive arrival epochs form a Poisson process with rate $\lambda$, and the generic batch size $X$ has the probability distribution $P\{X=k\}$ for $k\ge 1$ with a finite mean $\chi_1$.
If the server is free at the arrival epoch, then one of the arriving customers receives service immediately and the others join the orbit becoming repeated customers, whereas if the server is busy, all arriving customers join the orbit becoming repeated customers. Each of the repeated customers in the orbit independently repeatedly tries for receiving service after an exponential time with rate $\mu$ until success, or until it finds the server idle and then starts its service immediately. The customer in service leaves the system immediately after the completion of its service. Both primary and repeated customers require the same amount of the service time. Assume the generic service time $B$ has the probability distribution $B(x)$ with $B(0)=0$ with a finite mean $\beta_1$. Let $\rho=\lambda\beta_1 \chi_1$. It is well known that the system is stable if and only if (iff) $\rho<1$, which is assumed to hold throughout the paper.

We use $\beta(s)$ and $\beta_n$ to represent the LST and the $n$th moment of $B(x)$, respectively.
The generating function (GF) of $X$ is denoted by $X(z)=E(z^X)=\sum_{k=1}^{\infty}P\{X=k\} z^k$. In addition, we define $X_0=X-1$ and then it is clear that $X_0(z)=E(z^{X_0})=X(z)/z$.

Let $N_{orb}$ be the number of the repeated customers in the orbit, and $C_{ser}=1\mbox{ or}\ 0$ corresponds to the server being busy or idle, respectively.
Let $D^{(0)}$ ($D^{(1)}$) be a random variable (rv) having the same distribution as the conditional distribution of the number of repeated customers in the orbit given that the server is free (busy).
 It is clear that
$D^{(0)}$ takes nonnegative integers with the GF $D^{(0)}(z)=E(z^{D^{(0)}}) \stackrel{\rm def}{=} E(z^{N_{orb}}|C_{ser}=0)$. Note that $P\{C_{ser}=0\}=1-\rho$.
The following result on $D^{(0)}(z)$ (page 174 of Falin and Templeton~\cite{Falin-Templeton:1997}) is our start point:

\begin{eqnarray}\label{D{(0)}(z)}
D^{(0)}(z)=\exp\left\{-\frac {\lambda}{\mu}\int_z^1\frac {1-\beta(\lambda-\lambda X(u)) X_0(u)} {\beta(\lambda-\lambda X(u))-u} du\right\}.
\end{eqnarray}

Our particular interest is to analyze the asymptotic behavior of the tail probability for $D^{(0)}$ which is the independent increment from $L_{\infty}$ to $L_{\mu}$ in the stochastic decomposition, see, e.g., \cite{Yamamuro:2012} and also Section~\ref{sec:6}, from which the tail asymptotic behaviour (refined equivalence theorem) for the total number of customers is proved in Section~\ref{sec:6}, and the tail asymptotic behaviour for $D^{(1)}$
is also a consequence of the above asymptotic result (see Section~\ref{sec:7}).
To proceed, we first rewrite (\ref{D{(0)}(z)}). Let

\begin{eqnarray}
K^{\ast}(u)&=&\frac {1-\beta(\lambda-\lambda X(u)) X_0(u)} {(\rho+\chi_1 -1)(1-u)},\label{K{ast}(u)}\\
K^{\circ}(u)&=&\frac {(1-\rho)(1-u)} {\beta(\lambda-\lambda X(u))-u},\label{K{circ}(u)}\\
K(u)&=&K^{\ast}(u)\cdot K^{\circ}(u),\label{K(u)}\\
\psi &=&\frac{\lambda(\rho+\chi_1 -1)}{\mu(1-\rho)}.\label{psi}
\end{eqnarray}
It immediately follows from (\ref{D{(0)}(z)}) that
\begin{eqnarray}\label{D{(0)}(z)-2}
D^{(0)}(z)=\exp\left\{-\psi\int_z^1 K(u)  du\right\}.
\end{eqnarray}

The analysis of $D^{(0)}$ will be carried out in the following three sections: in Section~\ref{decomposition} we establish further stochastic decompositions for each of the two components (having GFs $K^{\ast}(u)$ and $K^{\circ}(u)$, respectively) in the decomposition of a random variable having the GF $K(u)$; in Section~\ref{sec:4}, asymptotic analysis on the components in the decomposition is carried out; and we complete the proof to the key result (the tail asymptotic behaviour of $D^{(0)}$) in Section~\ref{sec:5}.

\section{Stochastic decompositions related to $K(z)$}
\label{decomposition}

In this section, we first prove that both $K^{\ast}(z)$ and $K^{\circ}(z)$ are the GFs of the probability distributions for two discrete nonnegative random variables (rvs), denoted by $K^{\ast}$ and $K^{\circ}$, respectively. Assume that $K^{\ast}$ and $K^{\circ}$ are independent. Therefore, according to (\ref{K(u)}), $K(z)$ is the GF of $K=K^{\ast} + K^{\circ}$.
We then further decompose $K^{\ast}$ and $K^{\circ}$, respectively, into sums of independent rvs, for which we can carry out tail asymptotic analysis (given in the next section).

To see $K^{\ast}(z)$ is the GF for a probability distribution, we need to see the following: (1) $\beta(\lambda-\lambda X(z))$ is the GF for a discrete nonnegative random variable (rv), so is $\beta(\lambda-\lambda X(z)) X_0(u)$; and (2) for a GF $Q(z)$ of a discrete nonnegative rv, $1-Q(z)/(1-z)$ is essentially (by missing a constant) the GF of its equilibrium distribution. Specifically, we have the following facts (Facts~A--D).

\noindent{\bf Fact A:}
Let $N_{B}$ and $N_{B_X}$ be the number of batches and the total number of customers arrived within a service time $B$, respectively.
It is then clear that $N_{B_X}=X^{(1)}+X^{(2)}+\cdots+X^{(N_{B})}$, where $X^{(1)}, X^{(2)},\cdots, X^{(N_{B})}$ are independent copies of the batch size $X$.
It is well known that
\begin{eqnarray} \label{eqn:BX}
E(z^{N_{B}}) &=& \int_0^{\infty}\sum_{k=0}^{\infty}\frac {(\lambda x)^{k}}{k!}e^{-\lambda x}{z^k} dB(x)=\beta(\lambda-\lambda z),\\
E(z^{N_{B_X}}) &=& \int_0^{\infty}\sum_{k=0}^{\infty}\frac {(\lambda x)^{k}}{k!}e^{-\lambda x}(X(z))^k dB(x)=\beta(\lambda-\lambda X(z)).
\end{eqnarray}

Let $X_0$ and $N_{B_X}$ are independent, then $\beta(\lambda-\lambda X(z))  X_0(z)$ is the GF of $N_{B_XX_0}\stackrel{\rm def}{=} N_{B_X}+X_0$.

\noindent{\bf Fact B:}
$E(N_{B})=\lambda\beta_1$, $E(N_{B_X})=\lim_{z\uparrow 1}\frac {d}{dz}\beta(\lambda-\lambda X(z))=\lambda\beta_1\chi_1=\rho$, and for $N_{B_XX_0}$,
\begin{equation}
    E(N_{B_XX_0})=E(N_{B_X}+X_0)=\rho+\chi_1-1.\label{N_{B_XX_0}}
\end{equation}

\noindent{\bf Fact C:}
Let $N$ be an arbitrary discrete nonnegative rv with the GF $Q(z)=\sum_{n=0}^{\infty}q(n)z^n$, where
$q(n)=P\{N=n\}$. Denote by $\overline{q}(n)$ the tail probability of $N$, i.e.,
$\overline{q}(n)\stackrel{\rm def}{=}P\{N>n\}=\sum_{k=n+1}^{\infty}q(k)$, $n\ge 0$.
Under the assumption that $E(N)<\infty$, the discrete equilibrium probability distribution associated with $\{q(n)\}_{n=0}^{\infty}$ is defined by
\begin{equation}
    q^{(de)}(n)\stackrel{\rm def}{=}\overline{q}(n)/E(N)=P\{N>n\}/E(N).
\end{equation}
Let us use the notation $N^{(de)}$ to represent a rv having the distribution $\{q^{(de)}(n)\}_{n=0}^{\infty}$. Then, the GF of $\{q^{(de)}(n)\}_{n=0}^{\infty}$ is given by
\begin{eqnarray}\label{Q{(de)}(z)}
Q^{(de)}(z)&=&\displaystyle\frac {1}{E(N)}\cdot\frac {1-Q(z)}{1-z},
\end{eqnarray}
which follows from the fact that
\begin{eqnarray}\label{overline{Q}(z)}
\sum_{n=0}^{\infty}\left(\sum_{k=n+1}^{\infty}q(k)\right)z^n=\sum_{k=1}^{\infty}\sum_{n=0}^{k-1}q(k) z^n=\sum_{k=0}^{\infty}\frac {q(k) (1-z^k)} {1-z}
=\frac {1-Q(z)}{1-z}.
\end{eqnarray}

Now, according to (\ref{K{ast}(u)}) and the above Facts, we have
\begin{equation}
    K^{\ast}\stackrel{\rm def}{=}N_{B_XX_0}^{(de)},
\end{equation}
where the symbol
$\stackrel{\rm def}{=}$ means the equality in probability distribution, or $K^{\ast}(z)$ is the GF of the discrete equilibrium distribution of $N_{B_XX_0}$.
\bigskip

\noindent{\bf Fact D:} It can be shown that $K^{\circ}(z)$ is also the GF of a discrete probability distribution.
Let $B^{(e)}(x)$ be the equilibrium distribution of $B(x)$,
which is defined by $1-B^{(e)}(x)= \beta_1^{-1}\int_x^{\infty}(1-B(t))dt$. It is well known that the LST of $B^{(e)}(x)$ is
$\displaystyle\beta^{(e)}(s)=(1-\beta(s))/(\beta_1 s)$. Moreover, by Fact~C, we know that $X^{(de)}(z)\stackrel{\rm def}{=}(1-X(z))/(\chi_1(1-z))$
is the GF of a discrete nonnegative rv, denoted by $X^{(de)}$. Therefore (\ref{K{circ}(u)}) can be rewritten as
\begin{eqnarray}\label{K{circ}(z)-2}
K^{\circ}(z)&=&(1-\rho)\left[1-\frac {1-\beta(\lambda-\lambda X(z))}{1-X(z)} \cdot \frac {1-X(z)}{1-z}\right]^{-1}\nonumber\\
&=&\frac {1-\rho}{1-\rho\beta^{(e)}(\lambda-\lambda X(z))\cdot X^{(de)}(z)}\nonumber\\
&=&\sum_{k=0}^{\infty}(1-\rho)\rho^{k}\left(\beta^{(e)}(\lambda-\lambda X(z))\cdot X^{(de)}(z)\right)^k.
\end{eqnarray}
Let $B^{(e)}$ be a rv with probability distribution function $B^{(e)}(x)$.  Denote by $N_{B^{(e)}}$ and $N_{B^{(e)}_X}$ the number of batches and the total number of customers arriving within a random time $B^{(e)}$, respectively.
By Fact~A, we immediately know that $\beta^{(e)}(\lambda-\lambda X(z))$ is the GF of a discrete nonnegative rv, denoted by $N_{B^{(e)}_X}$. Therefore,
$\beta^{(e)}(\lambda-\lambda X(z))\cdot X^{(de)}(z)$ is the GF of
$N_{B^{(e)}_XX^{(de)}}\stackrel{\rm def}{=}N_{B^{(e)}_X}+X^{(de)}$, where $N_{B^{(e)}_X}$ and $X^{(de)}$
are independent.
From (\ref{K{circ}(z)-2}), $K^{\circ}$ can be viewed as the geometric sum of i.i.d. rvs, i.e.,
\begin{eqnarray}\label{K{circ}}
K^{\circ}=N_{B^{(e)}_XX^{(de)}}^{(1)}+N_{B^{(e)}_XX^{(de)}}^{(2)}+\cdots+N_{B^{(e)}_XX^{(de)}}^{(J)}\quad \mbox{for $J\ge 1$, and $K^{\circ}=0$ if $J=0$},
\end{eqnarray}
where $P(J=k)=(1-\rho)\rho^{k}$ ($k\ge 0$), rvs $N_{B^{(e)}_XX^{(de)}}^{(i)}$ ($i\ge 1$) are independent copies of $N_{B^{(e)}_XX^{(de)}}$, and $J$ and $N_{B^{(e)}_XX^{(de)}}^{(i)}$ ($i\ge 1$) are independent.

Finally, it follows from Facts~C and D, and the expression in (\ref{K(u)}) that $K$ can be regarded as the sum of independent rvs $K^{\ast}$ and $K^{\circ}$, i.e.,
\begin{eqnarray}\label{KK^{(a)}K^{(b)}}
K\stackrel{\rm def}{=}K^{\ast}+K^{\circ}
\end{eqnarray}
having the GF given in (\ref{K(u)}).

\section{Asymptotic tail probability for the rv $K$}
\label{sec:4}

In this subsection, we present tail asymptotic results for the components in the stochastic decompositions for $K^{\ast}$ and $K^{\circ}$, based on which our key result (Theorem~\ref{Theorem 3.2}) on the asymptotic tail behavior for $D^{(0)}$ is proved. For convenience of readers, a collection of literature results, required in this paper, are provided in Appendix~\ref{app-A}.

Throughout the rest of the paper, $R_{\sigma}$ and $S$ are the collections of the regularly varying (at $\infty$) functions with index $\sigma$ and subexponential functions, respectively, and $L(x)$ is a slowly varying (at $\infty$) function. Refer to Appendix~\ref{app-A} for more details.
It is also worthwhile to mention that for a distribution $F$ on $(0,\infty)$, if $1-F(x)\in R_{-\alpha}$ for $\alpha\ge 0$, then $F\in \mathcal S$
(see, e.g., Embrechts, Kluppelberg and Mikosch~\cite{Embrechts1997}).

Our discussion is based on the assumption that both service time $B$ and the batch size $X$ have regularly varying tails. Specifically, we make the following assumptions:
\begin{description}
\item[A1.] $P\{B>x\}\sim x^{-d_B}L(x)$ as $x\to\infty$ where $d_B>1$; and
\item[A2.] $P\{X>j\}\sim c_X\cdot j^{-d_X}L(j)$ as $j\to\infty$ where $d_X>1$ and $c_X\ge 0$.
\end{description}

\begin{remark}\label{Remark 3.2a}
It is a convention that in A2, $c_X=0$ means that
\[
    \lim_{j\to\infty}\frac {P\{X>j\}} {j^{-d_X}L(j)}=0.
\]
\end{remark}

By Karamata's theorem (e.g., page 28 in Bingham, Goldie and Teugels~\cite{Bingham:1989}) and the Assumption~A1, we know that
$\int_x^{\infty}(1-B(t))dt \sim (d_B-1)^{-1} x^{-d_B+1} L(x)$ as $x\to\infty$, which implies $1-B^{(e)}(x)\sim ((d_B-1)\beta_1)^{-1}  x^{-d_B+1} L(x)$ as $x\to\infty$.
By the definitions of $N_{B}$ and $N_{B^{(e)}}$ in Fact A and Lemma~\ref{Lemma 3.1}, we have
\begin{eqnarray}
P\{N_{B}>j\}&\sim& 1-B(j/\lambda)\ \sim\ \lambda^{d_B} j^{-d_B}L(j), \label{NB>j}\\
P\{N_{B^{(e)}}>j\}&\sim& 1-B^{(e)}(j/\lambda)\ \sim\ \frac {\lambda^{d_B-1}} {(d_B-1)\beta_1} j^{-d_B+1}L(j).\label{B(e)>j}
\end{eqnarray}

Next, let us state a result on tail asymptotics for $K$, which will be used in later sections.
\begin{theorem}\label{Theorem 3.1}
Under Assumptions~A1 and A2,
\begin{equation}
    P\{K>j\} \sim c_{K}\cdot j^{-a+1}L(j),\quad \mbox{as }j\to\infty,\label{Theorem3-1a}
\end{equation}
where $a=\min(d_B, d_X)>1$ and
\begin{equation}\label{c-K}
c_{K}=\left\{
\begin{array}{ll}
(\lambda\chi_1)^{a}\chi_1/((a-1)(1-\rho)(\rho+\chi_1-1)),&\quad \mbox{if}\ d_X>d_B>1,\\
c_X/((a-1)(1-\rho)(\rho+\chi_1-1)),&\quad \mbox{if}\ 1<d_X<d_B \mbox{ and}\ c_X>0, \\
((\lambda\chi_1)^{a}\chi_1+c_X)/((a-1)(1-\rho)(\rho+\chi_1-1)),&\quad \mbox{if}\ d_X=d_B>1 \mbox{ and}\ c_X>0.
\end{array}
\right.
\end{equation}
\end{theorem}
Based on whether or not the batch size $X$ has a tail lighter than the service time $B$, we divided our proof to Theorem~\ref{Theorem 3.1} into the following three cases.

\subsection{Case 1: $d_X>d_B>1$ in Assumptions~A1 and A2}

This is the case, in which the batch size $X$ has a tail lighter than the service time $B$. It is worthwhile to mention that in this case
$X$ is not necessarily light-tailed (see, e.g., Grandell \cite{Grandell1997}, p.146).

\begin{lemma}\label{Lemma 3.5}
If $d_X>d_B>1$ in Assumptions~A1 and A2, then as $j\to\infty$,
\begin{eqnarray}
P\{X>j\}&=&o(j^{-d_B}L(j)), \label{Xo-1}\\
P\{X_0>j\}&=&o(j^{-d_B}L(j)), \label{Xo-2}\\
P\{X^{(de)}>j\}&=&o(j^{-d_B+1}L(j)).\label{Xo-3}
\end{eqnarray}
\end{lemma}
\begin{proof} Because of $d_X>d_B$, (\ref{Xo-1}) and (\ref{Xo-2}) directly follow from Assumptions~A1 and A2.
We now prove (\ref{Xo-3}). By Assumption~A2, $P\{X>j\}\le c'_X j^{-d_X}L(j)$ for some $c'_X>0$. Since
$P\{X^{(de)}=j\}=P\{X>j\}/\chi_1$ (by the definition of the equilibrium distribution),
\begin{eqnarray}
P\{X^{(de)}>j\}
\le (c'_X/\chi_1)\sum_{k=j+1}^{\infty} k^{-d_X}L(k)
\ \sim\ \frac {c'_X/\chi_1} {d_X-1} j^{-d_X+1}L(j)\quad \mbox{(by Lemma \ref{Lemma 3.3})}
\end{eqnarray}
which leads to (\ref{Xo-3}) due to $d_X>d_B$.
\end{proof}

By (\ref{NB>j}), (\ref{B(e)>j}), (\ref{Xo-1}) and (\ref{Xo-3}), we immediately have $P\{X>j\}=o(P\{N_{B}>j\})$ and $P\{X^{(de)}>j\}=o(P\{N_{B^{(e)}}>j\})$. By the definitions of $N_{B_X}$ and $N_{B^{(e)}_X}$ in Facts A and D, and applying  Lemma \ref{Lemma 3.2}, we have
\begin{eqnarray}
P\{N_{B_X}>j\}&\sim& (\lambda\chi_1)^{d_B} j^{-d_B}L(j), \label{NBX>n}\\
P\{N_{B^{(e)}_X}>j\}&\sim&\frac {(\lambda\chi_1)^{d_B-1}} {(d_B-1)\beta_1} j^{-d_B+1}L(j).\label{NBeX>n}
\end{eqnarray}

By the definitions in Facts B and D, $N_{B_XX_0}=N_{B_X}+X_0$ and $N_{B^{(e)}_XX^{(de)}}=N_{B^{(e)}_X}+X^{(de)}$, and (\ref{NBX>n}) and (\ref{NBeX>n}) lead to $P\{X_0>j\}=o(P\{N_{B_X}>j\})$ and $P\{X^{(de)}>j\}=o(P\{N_{B^{(e)}_X}>j\})$ due to $d_X>d_B$.
Applying Part (i) of Lemma \ref{Lemma 3.4new}, we have
\begin{eqnarray}
P\{N_{B_XX_0}>j\}&\sim& P\{N_{B_X}>j\}\ \sim\ (\lambda\chi_1)^{d_B} j^{-d_B}L(j), \label{NBXX0-NBX}\\
P\{N_{B^{(e)}_XX^{(de)}}>j\}&\sim& P\{N_{B^{(e)}_X}>j\}\ \sim\  \frac {(\lambda\chi_1)^{d_B-1}} {(d_B-1)\beta_1} j^{-d_B+1}L(j).\label{NBeXXe-NBeX}
\end{eqnarray}

Now we are ready to present the asymptotic property for the tail probability of $K$.
By (\ref{K{circ}}) and (\ref{NBeXXe-NBeX}), and applying Corollary A3.21 in Embrechts, et al. (\cite{Embrechts1997}, p. 581),
\begin{eqnarray}
P\{K^{\circ}>j\}=\frac {\rho} {1-\rho}P\{N_{B^{(e)}_XX^{(de)}}>j\}&\sim& \frac {(\lambda\chi_1)^{d_B}} {(d_B-1)(1-\rho)} j^{-d_B+1}L(j), \label{K^{(b)}>n}
\end{eqnarray}
where in the first equality we have used the fact that $\rho/(1-\rho)$ is the mean of rv $J$ in (\ref{K{circ}}).
By Facts B and C, and (\ref{NBXX0-NBX}),
\[
P\{K^{\ast}=j\}=P\{N_{B_XX_0}^{(de)}=j\}=\frac {P\{N_{B_XX_0}>j\}} {E(N_{B_XX_0})}\sim
\frac {(\lambda\chi_1)^{d_B} }{\rho+\chi_1-1} j^{-d_B}L(j). 
\]
Applying Lemma \ref{Lemma 3.3} gives
\begin{eqnarray}
P\{K^{\ast}>j\}&\sim& \frac {(\lambda\chi_1)^{d_B}} {(d_B-1)(\rho+\chi_1-1)} j^{-d_B+1}L(j).\label{K^{(a)}>n}
\end{eqnarray}
By (\ref{KK^{(a)}K^{(b)}}), (\ref{K^{(a)}>n}) and (\ref{K^{(b)}>n}) and using Part (ii) of Lemma \ref{Lemma 3.4new}, we have
\begin{eqnarray}
P\{K>j\}&\sim& \frac {(\lambda\chi_1)^{d_B}\chi_1} {(d_B-1)(1-\rho)(\rho+\chi_1-1)}\cdot j^{-d_B+1}L(j),\label{P{K>j}}
\end{eqnarray}
which is the conclusion in Theorem~\ref{Theorem 3.1} for Case~1.

\subsection{Case 2: $1<d_X<d_B$ and $c_X> 0$ in Assumptions~A1 and A2}

This is the case, in which the batch size $X$ has a tail heavier than the service time $B$.
By the definitions of $N_{B}$, $N_{B_X}$ and $N_{B_XX_0}$ in Facts A and B, and applying Lemma \ref{Lemma 3.2} and Part (ii) of Lemma \ref{Lemma 3.4new}, we have
\begin{eqnarray}
P\{N_{B_X}>j\}&\sim& \lambda\beta_1\cdot c_X j^{-d_X}L(j), \label{NBX>n-case2}\\
P\{N_{B_XX_0}>j\}&\sim& (1+\lambda\beta_1)\cdot c_X j^{-d_X}L(j), \label{NBXX0-NBX-case2}
\end{eqnarray}
where we have used the facts $E(N_B)=\lambda\beta_1$ and $P\{X_0>j\}\sim P\{X>j\}$.

By (\ref{B(e)>j}) and the definition of $N_{B^{(e)}_X}$ in Fact D, and applying Lemma \ref{Lemma 3.2},
\begin{eqnarray}
P\{N_{B^{(e)}_X}>j\}&\le& c''_X \max\left(j^{-d_B+1}L(j),\ j^{-d_X}L(j)\right)\quad\mbox{for some\ }c''_X>0.\label{NBeX>n-case2}
\end{eqnarray}
By Lemma \ref{Lemma 3.3}, we have $P\{X^{(de)}>j\}\sim (\chi_1(d_X-1))^{-1} c_X j^{-d_X+1}L(j)$, which implies $P\{N_{B^{(e)}_X}>j\}=o(P\{X^{(de)}>j\})$.
By the definition $N_{B^{(e)}_XX^{(de)}}$ in Fact D, and applying Part (i) of Lemma \ref{Lemma 3.4new}, we get
\begin{eqnarray}
P\{N_{B^{(e)}_XX^{(de)}}>j\}\sim P\{X^{(de)}>j\}\sim \frac {c_X} {\chi_1(d_X-1)} j^{-d_X+1}L(j).\label{NBeXXe-NBeX-case2}
\end{eqnarray}

Now we are ready to present the asymptotic property for the tail probability of $K$.
By (\ref{K{circ}}) and (\ref{NBeXXe-NBeX-case2}), and applying Corollary A3.21 in Embrechts, et al. (\cite{Embrechts1997}, p. 581),
\begin{eqnarray}
P\{K^{\circ}>j\}=\frac {\rho} {1-\rho}P\{N_{B^{(e)}_XX^{(de)}}>j\}&\sim& \frac {\lambda\beta_1 c_X} {(1-\rho)(d_X-1)} j^{-d_X+1}L(j).\label{K^{(b)}>n-case2}
\end{eqnarray}
By Facts B and C, and (\ref{NBXX0-NBX-case2}),
\[
P\{K^{\ast}=j\}=P\{N_{B_XX_0}^{(de)}=j\}=\frac {P\{N_{B_XX_0}>j\}} {E(N_{B_XX_0})}\sim \frac {(1+\lambda\beta_1)c_X} {\rho+\chi_1-1} j^{-d_X}L(j). 
\]
Applying Lemma \ref{Lemma 3.3},
\begin{eqnarray}
P\{K^{\ast}>j\}\sim \frac {(1+\lambda\beta_1)c_X} {(d_X-1)(\rho+\chi_1-1)}j^{-d_X+1}L(j).\label{K^{(a)}>n-case2}
\end{eqnarray}
By (\ref{KK^{(a)}K^{(b)}}), (\ref{K^{(a)}>n-case2})--(\ref{K^{(b)}>n-case2}) and using Part (ii) of Lemma~\ref{Lemma 3.4new},
\begin{eqnarray}
P\{K>j\}&\sim& \frac {c_X} {(d_X-1)(1-\rho)(\rho+\chi_1-1)}\cdot j^{-d_X+1}L(j), \label{P{K>j}-case2}
\end{eqnarray}
which is the conclusion in Theorem~\ref{Theorem 3.1} for Case~2.

\subsection{Case 3: $d_X=d_B=a>1$ and $c_X> 0$ in Assumptions~A1 and A2}

This is the case, in which the batch size $X$ has a tail equivalent to the service time $B$. Following the same procedure in Cases 1 and 2, we can prove that
\begin{eqnarray}
P\{N_{B_X}>j\}&\sim& ((\lambda\chi_1)^{a}+\lambda\beta_1c_X )\cdot  j^{-a}L(j), \label{NBX>n-case3}\\
P\{N_{B^{(e)}_XX^{(de)}}>j\}&\sim& \frac {(\lambda\chi_1)^{a}+\lambda\beta_1c_X} {(a-1)\rho}\cdot j^{-a+1}L(j), \label{N^{e}BXX^{de}>n-case3}\\
P\{K^{\circ}>j\}&\sim& \frac {(\lambda\chi_1)^a+\lambda\beta_1 c_X} {(a-1)(1-\rho)}\cdot  j^{-a+1}L(j),\label{K^{(b)}>n-case3}\\
P\{K^{*}>j\}&\sim& \frac {(\lambda\chi_1)^a+(1+\lambda\beta_1) c_X} {(a-1)(\rho+\chi_1-1)}\cdot  j^{-a+1}L(j),\label{K^{*}>n-case3}\\
P\{K>j\}&\sim& \frac {(\lambda\chi_1)^{a}\chi_1+c_X} {(a-1)(1-\rho)(\rho+\chi_1-1)}\cdot j^{-a+1}L(j),
\end{eqnarray}
where we have skipped the detailed derivations to avoid the repetition.

\section{Key result -- asymptotic tail probability for the rv $D^{(0)}$}
\label{sec:5}

Note that $D^{(0)}(z)$ is explicitly expressed by $K(z)$ in (\ref{D{(0)}(z)-2}), based on which we are able to study the asymptotic property for the tail probability of $D^{(0)}$ using the result on $K$ in Theorem~\ref{Theorem 3.1}. This is  the key result of this paper since the refined asymptotic properties in the main theorem (Theorem~\ref{Theorem refined}) and the asymptotic property of $D^{(1)}$ in Theorem~\ref{Theorem 5.1}, can be readily proved by using the following Theorem~\ref{Theorem 3.2}

\begin{theorem}[Key result] \label{Theorem 3.2}
Under Assumptions~A1 and A2,
\begin{eqnarray}
P\{D^{(0)}>j\}\sim (1-1/a) c_K\psi\cdot j^{-a}L(j)=c_{D^{(0)}}\cdot j^{-a}L(j),\quad\mbox{as}\ j\to\infty,\label{main-1}
\end{eqnarray}
where $a=\min(d_B, d_X)>1$,
\begin{eqnarray}\label{main-2}
c_{D^{(0)}}&=&\left\{
\begin{array}{ll}
(\lambda\chi_1)^{a+1}/(a\mu(1-\rho)^2), &\quad \mbox{if}\ d_X>d_B>1,\\
\lambda c_X/(a\mu(1-\rho)^2), &\quad \mbox{if}\ 1<d_X<d_B \mbox{ and}\ c_X>0, \\
((\lambda\chi_1)^{a+1}+\lambda c_X)/(a\mu(1-\rho)^2), &\quad \mbox{if}\ d_X=d_B>1 \mbox{ and}\ c_X>0,
\end{array}
\right.
\end{eqnarray}
and $\psi$ and $c_K$ are expressed in (\ref{psi}) and (\ref{c-K}), respectively.
\end{theorem}

Once again, we put some literature results required in the proof to our main theorem, together with some preliminary properties, in Appendix~\ref{app-A}.

In the following, we divide the proof to Theorem~\ref{Theorem 3.2} into two parts, depending on whether $a$ is an integer or not.
First let us rewrite (\ref{D{(0)}(z)-2}) as follows:
\begin{equation}
    D^{(0)}(z) = 1 -\psi \int_z^1 K(u)du +\sum_{k=2}^{\infty}{(-\psi)^k \over k!}\left(\int_z^1 K(u)du\right)^k. \label{D{(0)}(z)new}
\end{equation}
As shown in Facts~A--D, $K(z)$ is the GF of the rv $K$ with the discrete probability distribution $k(j)\stackrel{\rm def}{=}P\{K=j\}$, $j\ge 0$. In the proof, we use
the notation $\kappa_n$ to represent the $n$th factorial moment (see Appendix~\ref{app-A} for the definition) of $K$.

\subsection{Proof for the non-integer $a>1$}

Suppose $m<a<m+1$, $m\in\{1,2,\cdots\}$. By Theorem \ref{Theorem 3.1}, $P\{K>j\}\sim c_K\cdot j^{-a+1}L(j)$.
So $\kappa_{m-1}<\infty$ and $\kappa_{m}=\infty$.

Define $K_{m-1}(z)$ in a manner similar to that in (\ref{G_n(z)}). Corresponding to the sequence $\{k(j)\}_{j=0}^{\infty}$, we also define $\overline{k}_n(j)$, $n\in\{0,1,\cdots m-1\}$ in a way similar to that in (\ref{overline{g}_{0}(j)}) and (\ref{overline{g}_{n+1}(j)}). Note that $\overline{k}_1(j)=P\{K>j\}\sim c_K\cdot j^{-a+1}L(j)$. By Lemma~\ref{Lemma 4.3},
\begin{eqnarray}
K_{m-1}(z)&\sim&\frac{\Gamma(a-m)\Gamma(m+1-a)} {\Gamma(a-1)} c_K(1-z)^{a-1}L(1/(1-z)),\quad z\uparrow 1.\label{(a>1)4.2-10}
\end{eqnarray}
By Karamata's theorem (Bingham, Goldie and Teugels~\cite{Bingham:1989}, p.28),
\begin{eqnarray}
\int_z^{1}K_{m-1}(u)du &\sim& \frac{\Gamma(a-m)\Gamma(m+1-a)} {\Gamma(a-1)a} c_K(1-z)^{a}L(1/(1-z)),\quad z\uparrow 1.\label{(a>1)4.2-11}
\end{eqnarray}
Next, we present a relation between $D^{(0)}_m(z)$ and $K_{m-1}(z)$. By the definition of $K_{m-1}(z)$,
\begin{eqnarray}
K(z)&=&\sum_{k=0}^{m-1}(-1)^k\frac{\kappa_k}{k!}(1-z)^k+(-1)^{m}K_{m-1}(z),\label{5-1K12}\\
\int_z^1K(u)du&=&-\sum_{k=1}^{m}(-1)^k\frac{\kappa_{k-1}}{k!}(1-z)^k+(-1)^{m}\int_z^{1}K_{m-1}(u)du.\label{(a>1)4.2-12}
\end{eqnarray}
Note that $\int_z^{1}K_{m-1}(u)du/(1-z)^{m}\to 0$ and $\int_z^{1}K_{m-1}(u)du/(1-z)^{m+1}\to\infty$ as $z\uparrow 1$.

From (\ref{D{(0)}(z)new}) and (\ref{(a>1)4.2-12}), there are constants $\{v_k;\ k=0,1,2,\cdots,m\}$ satisfying
\begin{eqnarray}
D^{(0)}(z)&=&\sum_{k=0}^{m}(-1)^k v_k (1-z)^k+(-1)^{m+1}\psi\int_z^{1}K_{m-1}(u)du + O((1-z)^{m+1}),\quad z\uparrow 1. \label{(a>1)4.2-14}
\end{eqnarray}
Define $D^{(0)}_m(z)$ in a manner similar to that in (\ref{G_n(z)}). By (\ref{(a>1)4.2-14}),
\begin{eqnarray}
D^{(0)}_m(z)&=&\psi\int_z^{1}K_{m-1}(u)du + O((1-z)^{m+1})\nonumber\\
&\sim&\psi\int_z^{1}K_{m-1}(u)du,\quad z\uparrow 1.\label{(a>1)4.2-15}
\end{eqnarray}
By (\ref{(a>1)4.2-11}) and (\ref{(a>1)4.2-15}),
\begin{eqnarray}
D^{(0)}_m(z)&\sim& \frac{\Gamma(a-m)\Gamma(m+1-a)} {\Gamma(a)}\cdot\frac {(a-1)c_K\psi}{a}(1-z)^{a}L(1/(1-z)),\quad z\uparrow 1.\label{(a>1)4.2-16}
\end{eqnarray}
By applying Lemma \ref{Lemma 4.3},
\begin{eqnarray}
P\{D^{(0)}>j\}&\sim& \frac {(a-1)c_K\psi}{a} j^{-a}L(j),\quad j\to\infty,\label{(a>1)4.2-17}
\end{eqnarray}
which completes the proof of Theorem~\ref{Theorem 3.2} for non-integer $a>1$.

\subsection{Proof for the integer $a>1$}

Suppose $a=m\in\{2,3,\cdots\}$.  By Theorem~\ref{Theorem 3.1}, $P\{K>j\}\sim c_K\cdot j^{-m+1}L(j)$.
So, $\kappa_{m-2}<\infty$. Unfortunately, whether $\kappa_{m-1}$ is finite or not remains uncertain, which is determined essentially by whether $\sum_{k=1}^{\infty}k^{-1}L(k)$ is convergent or not. For this reason we have to sharpen our analytical tool by
introducing the de Haan class $\Pi$ of slowly varying functions (see Definition~\ref{II}).

\begin{lemma}\label{Lemma 4.4new}
Suppose that
$\{q(j)\}_{j=0}^{\infty}$ is a nonnegative sequence with the GF $Q(z)$. The following two statements are equivalent:
\begin{flalign}
\begin{split}
\mbox{(i)}&\quad q(j)\sim j^{-1}L(j),\quad j\to\infty; \mbox{ and } \label{Lemma4-4a}
\end{split}&\\
\begin{split}
\mbox{(ii)}&\quad Q(1-u)\in \Pi \mbox{ at $0$ with an auxiliary function which can be taken as $L(1/u)$}. \label{Lemma4-4b}
\end{split}&
\end{flalign}
\end{lemma}
\begin{proof}
Let $r(j)=\sum_{k=0}^{j}k q(k)$, $j\ge 0$ and $R(z)=\sum_{j=0}^{\infty}r(j)z^j$. Noting that $r(0)=0$, we have
\begin{eqnarray}
R(z)&=&\sum_{j=1}^{\infty}\sum_{k=1}^{j}kq(k)z^j=\sum_{k=1}^{\infty}\sum_{j=k}^{\infty}kq(k)z^j=\sum_{k=1}^{\infty}kq(k)z^k/(1-z)=zQ'(z)/(1-z).
\end{eqnarray}
Therefore, for $x>0$,
\begin{eqnarray}
Q(1-xu)-Q(1-u)&=&\int_{1-u}^{1-xu}s^{-1}(1-s)R(s) ds\nonumber\\
&=&-\int_{1}^{x}(1-ut)^{-1}u^2tR(1-ut)dt.
\end{eqnarray}
Clearly, (\ref{Lemma4-4a}) is equivalent to $r(j)\sim jL(j)$,
Note that $\{r(j)\}_0^{\infty}$ is an increasing sequence. So, it follows from
Lemma \ref{Lemma 4.2} that (\ref{Lemma4-4a}) is equivalent to $R(z)\sim (1-z)^{-2}L(1/(1-z)),\ z\uparrow 1$, this is
$R(1-u)\sim u^{-2}L(1/u),\ u\downarrow 0$.
\begin{eqnarray}
\lim_{u\downarrow 0}\frac {Q(1-xu)-Q(1-u)}{L(1/u)}&=&-\int_{1}^{x}\lim_{u\downarrow 0}\frac {(1-ut)^{-1}u^2tR(1-ut)} {L(1/u)}dt\nonumber\\
&=&-\int_{1}^{x}1/t dt=-\log x,\label{Lemma4-4c}
\end{eqnarray}
where in the first equality we have used the uniform convergence theorem (see, e.g., Bingham, Goldie and Teugels~\cite{Bingham:1989}, p.22) on regular varying functions for interchanging the limit and the integration.
\end{proof}
\begin{lemma}\label{Lemma 4.5new}
Let $\{g(j)\}_{j=0}^{\infty}$ be a discrete probability distribution with the GF $G(z)$, and $n\in\{1,2,\cdots\}$, the following two statements are equivalent:
\begin{flalign}
\begin{split}
\mbox{(i)}&\quad \overline{g}_1(j)\sim j^{-n}L(j)\quad \mbox{ as }  j\to\infty; \label{deHaan3}
\end{split}&\\
\begin{split}
\mbox{(ii)}&\quad \lim_{u\downarrow 0}\frac {\widehat{G}_{n-1}(1-xu)-\widehat{G}_{n-1}(1-u)}{L(1/u)/(n-1)!}=-\log x\quad\mbox{ for all }x>0.\label{deHaan4}
\end{split}&
\end{flalign}
\end{lemma}
\begin{proof}
By using Lemma \ref{Lemma 3.3} repeatedly, (\ref{deHaan3}) is equivalent to
\begin{eqnarray}
\overline{g}_{n}(j)\sim j^{-1}L(j)/(n-1)!\quad \mbox{ as }  j\to\infty.\label{Lemma 4.5newproof-1}
\end{eqnarray}
Note that the sequence $\{\overline{g}_{n}(j)\}_{j=0}^{\infty}$ has the GF $\widehat{G}_{n-1}(z)$ (by Lemma \ref{Lemma 4.1}). The equivalence of
 (\ref{deHaan4}) and (\ref{Lemma 4.5newproof-1}) is proved by applying Lemma \ref{Lemma 4.4new}.
\end{proof}

\vskip 0.2cm

Since $\kappa_{m-2}<\infty$, we can define $K_{m-2}(z)$ in a manner similar to that in (\ref{G_n(z)}).
\begin{eqnarray}
K(z)&=&\sum_{k=0}^{m-2}(-1)^k\frac{\kappa_k}{k!}(1-z)^k+(-1)^{m-1}K_{m-2}(z),\label{(a=m)result1}
\end{eqnarray}
where $K_{m-2}(z)=o\left((1-z)^{m-2}\right)$ as $z\uparrow 1$.
\begin{eqnarray}
\int_z^{1}K(u)du&=&-\sum_{k=1}^{m-1}(-1)^k\frac{\kappa_{k-1}}{k!}(1-z)^k+(-1)^{m-1}\int_z^{1}K_{m-2}(u)du,\label{(a=m)result2}
\end{eqnarray}
where $\int_z^{1}K_{m-2}(u)du=o((1-z)^{m-1})$ as $z\uparrow 1$.

It follows from (\ref{D{(0)}(z)new}) and (\ref{(a=m)result2}) that for some constants $\{v_k;\ k=0,1,2,\cdots,m\}$,
\begin{eqnarray}
D^{(0)}(z)&=&\sum_{k=0}^{m}(-1)^kv_k (1-z)^k+(-1)^{m} \psi\int_z^{1}K_{m-2}(u)du +o((1-z)^{m}), \quad z\uparrow 1.
\end{eqnarray}
Define $\widehat{D}^{(0)}_{m-1}(z)$ in a manner similar to that in (\ref{widehat{G}_n (z)}), we have
\begin{eqnarray}
\widehat{D}^{(0)}_{m-1}(z)=v_m+ \frac {\psi} {(1-z)^m}\int_z^{1} (1-u)^{m-1}\widehat{K}_{m-2}(u)du + o(1),\quad z\uparrow 1,\label{(a=m)result3}
\end{eqnarray}
which immediately leads to:
\begin{eqnarray}
\widehat{D}^{(0)}_{m-1}(1-w)&=&v_m+ \frac {\psi} {w^m}\int_0^{w} u^{m-1}\widehat{K}_{m-2}(1-u)du + o(1),\quad w\downarrow 0,\label{(a=m)result4a}\\
\widehat{D}^{(0)}_{m-1}(1-xw)&=&v_m+ \frac {\psi} {(xw)^m}\int_0^{xw} u^{m-1}\widehat{K}_{m-2}(1-u)du + o(1)\nonumber\\
&=&v_m+ \frac {\psi} {w^m}\int_0^{w} u^{m-1}\widehat{K}_{m-2}(1-xu)du + o(1),\quad w\downarrow 0.\label{(a=m)result4b}
\end{eqnarray}
By (\ref{(a=m)result4a}) and (\ref{(a=m)result4b})
\begin{eqnarray}
&&\widehat{D}^{(0)}_{m-1}(1-xw)-\widehat{D}^{(0)}_{m-1}(1-w)\nonumber\\
&&\qquad\qquad\qquad =\frac {\psi} {w^m}\int_0^{w} u^{m-1}\left(\widehat{K}_{m-2}(1-xu)-\widehat{K}_{m-2}(1-u)\right)du + o(1)\quad w\downarrow 0.\label{(a=m)result4c}
\end{eqnarray}
Note that $\overline{k}_1(j)=P\{K>j\}\sim c_K\cdot j^{-m+1}L(j)$. By Lemma \ref{Lemma 4.5new}, we obtain
\begin{eqnarray}
\widehat{K}_{m-2}(1-xu)-\widehat{K}_{m-2}(1-u)\sim
-(\log x)c_KL(1/u)/(m-2)! \quad u\downarrow 0.\label{(a=m)result5a}
\end{eqnarray}
By Karamata's theorem (Bingham, Goldie and Teugels~\cite{Bingham:1989}, p.28), we know
\begin{eqnarray}
\int_0^{w} u^{m-1}\left(\widehat{K}_{m-2}(1-xu)-\widehat{K}_{m-2}(1-u)\right)du\sim
-(\log x)\frac {c_K} {m}w^m L(1/w)/(m-2)! \quad w\downarrow 0.\label{(a=m)result5b}
\end{eqnarray}
Therefore,
\begin{eqnarray}
\lim_{w\downarrow 0}\frac {\widehat{D}^{(0)}_{m-1}(1-xw)-\widehat{D}^{(0)}_{m-1}(1-w)} {L(1/w)/(m-1)!}
=-\frac{m-1} {m}c_K\psi \log x \quad \mbox{(by (\ref{(a=m)result4c}) and (\ref{(a=m)result5b}))}.\label{(a=m)result6}
\end{eqnarray}
By applying Lemma \ref{Lemma 4.5new}, we obtain from (\ref{(a=m)result6}) that
\begin{eqnarray}
p\{D^{(0)}>j\}&\sim& \frac{m-1} {m}c_K\psi j^{-m}L(j)\quad \mbox{as}\ j\to\infty,
\end{eqnarray}
which completes the proof of Theorem~\ref{Theorem 3.2} for integer $a=m\in\{2,3,\cdots\}$.

\section{Refined equivalence theorem}
\label{sec:6}
In this section, under assumptions A1 and A2 we first present the asymptotic tail equivalence for the total numbers of customers in an $M^X/G/1$ retrial queue and the corresponding standard $M^X/G/1$ queue without retrial, which is a generalization (under the assumption of regularly varying tails) of the equivalence theorem in the literature since we removed the restriction imposed on the batch, by allowing the batch size to have a tail probability heavier than that of the service time. Then, we focus on the difference between the tail probability of the total number of customers in the system for the retrial queue and the tail probability of the total number of customers in the corresponding non-retrial queue, and provide a characterization for the asymptotic behavior of this difference, which is our main contribution: a refined result for the tail equivalence between the two systems.

As mentioned in the introduction, in order to establish the equivalence theorem for a retrial queueing system, people often use a stochastic decomposition result (e.g., \cite{Shang-Liu-Li:2006}, \cite{Yamamuro:2012} and \cite{Masuyama:2014}). For the $M^X/G/1$ retrial queue, the total number $L_{\mu}$ of customers in the system can be written as the sum of two independent random variables, the total number $L_{\infty}$ of customers in the corresponding $M^X/G/1$ queueing system (without retrial) and $D^{(0)}$, i.e.,
\begin{eqnarray}\label{L=L+D}
L_{\mu}\stackrel{\rm def}{=}L_{\infty}+D^{(0)}.
\end{eqnarray}
It is well known that
\begin{eqnarray}\label{without-retrial}
Ez^{L_{\infty}}&=&\beta(\lambda-\lambda X(z))
\cdot \frac {(1-\rho)(1-z)} {\beta(\lambda-\lambda X(z))-z}.
\end{eqnarray}
The equality (\ref{L=L+D}) can be verified easily because
\begin{eqnarray}\label{p0(z)+zp1(z)}
Ez^{L_{\mu}}&=&\sum_{n=0}^{\infty}z^nP\{C_{ser}=0,N_{orb}=n\}+\sum_{n=0}^{\infty}z^{n+1}P\{C_{ser}=1,N_{orb}=n\}\nonumber\\
&=&p_0(z)+zp_1(z),
\end{eqnarray}
where $p_i(z)\stackrel{\rm def}{=}\sum_{n=0}^{\infty}z^nP\{C_{ser}=i,N_{orb}=n\}$, $i=0,1$, are explicitly expressed on page 174 of Falin and Templeton~\cite{Falin-Templeton:1997}, with which (\ref{p0(z)+zp1(z)}) leads to $Ez^{L_{\mu}}=Ez^{L_{\infty}}\cdot Ez^{D^{(0)}}$ and then (\ref{L=L+D}).

It follows from (\ref{without-retrial}) and (\ref{K{circ}(u)}) that
\begin{eqnarray} \label{eqn:6.4}
E(z^{L_{\infty}})&=&\beta(\lambda-\lambda X(z))\cdot K^{\circ}(z),
\end{eqnarray}
which implies that
$L_{\infty} \stackrel{\rm def}{=} N_{B_X}+K^{\circ}$,
where $N_{B_X}$ and $K^{\circ}$ are assumed to be independent.
Note that, from (\ref{K^{(b)}>n}), (\ref{K^{(b)}>n-case2}) and (\ref{K^{(b)}>n-case3}), under Assumptions~A1 and A2,
\begin{eqnarray}\label{K-circ-1}
    P\{K^{\circ}>j\}&\sim& c_{K^{\circ}}\cdot j^{-a+1}L(j),\quad\mbox{as }j\to\infty,
\end{eqnarray}
where $a=\min(d_B, d_X)>1$ and
\begin{eqnarray}\label{c-K-circ}
c_{K^{\circ}}&=&\left\{
\begin{array}{ll}
(\lambda\chi_1)^{a}/((a-1)(1-\rho)),&\quad \mbox{if }d_X>d_B>1,\\
\lambda\beta_1c_X/((a-1)(1-\rho)),&\quad \mbox{if }1<d_X<d_B \mbox{ and}\ c_X>0, \\
((\lambda\chi_1)^{a}+\lambda\beta_1c_X)/((a-1)(1-\rho)),&\quad \mbox{if }d_X=d_B>1 \mbox{ and}\ c_X>0.
\end{array}
\right.
\end{eqnarray}
It follows from (\ref{NBX>n}), (\ref{NBX>n-case2}), (\ref{NBX>n-case3}) and (\ref{K-circ-1}) that $P\{N_{B_X}>j\}=o(P\{K^{\circ}>j\})$. So,
\begin{eqnarray}
P\{L_{\infty}>j\}\sim P\{K^{\circ}>j\}\sim c_{K^{\circ}}\cdot j^{-a+1}L(j).\label{Linfinity-6.7}
\end{eqnarray}
By Theorem \ref{Theorem 3.2}, we have $P\{D^{(0)}>j\}=o(P\{L_{\infty}>j\})$, and therefore
\begin{eqnarray}\label{Lmu-Linfinity}
P\{L_{\mu}>j\}\sim P\{L_{\infty}>j\}.
\end{eqnarray}

Next, we refine the asymptotic equivalence (\ref{Lmu-Linfinity}). Precisely, we will characterize the asymptotic behavior of the difference $P\{L_{\mu}>j\}-P\{L_{\infty}>j\}$ as $j\to\infty$. Towards this end, we provide the following lemma, which will be used to confirm our assertion later.
We use the notation $\overline{F}(\cdot)=1-F(\cdot)$.

\begin{lemma}\label{lemma-F1F2}
    Let $X_1$ and $X_2$ be independent nonnegative rvs with distribution functions $F_1$ and $F_2$, respectively.
Assume that $\overline{F}_1(t)\sim c_1 t^{-d+1}L(t)$ and $\overline{F}_2(t)\sim c_2 t^{-d}L(t)$ as $t\to\infty$, where $c_1>0$, $c_2>0$, $d>1$ and $L(t)$ is a slowly varying function at infinity. Suppose that $\overline{F}_1(t)=\int_t^{\infty}f_1(x)dx$ and $f_1(x)$ is ultimately decreasing.
Then, as $t\to\infty$,
\begin{eqnarray}\label{lemma-2nd-order}
P\{X_1+X_2>t\}=\overline{F}_1(t)+ (d-1)\mu_{F_2}\cdot t^{-1}\overline{F}_1(t) (1+o(1))+\overline{F}_2(t) (1+o(1)),
\end{eqnarray}
where $\mu_{F_2}<\infty$ is the mean value of $X_2$.
\end{lemma}

\begin{proof}
For $0<\varepsilon<1$ and $t>0$, we have
\begin{eqnarray}\label{lemma-2nd-order-proof-A}
&&P\{X_1+X_2>t\}\nonumber\\
&=&P\{X_1>(1-\varepsilon )t,X_2>\varepsilon t\}\nonumber\\
&&+P\{X_1+X_2>t,X_2\le \varepsilon t\}+P\{X_1+X_2>t,X_1\le (1-\varepsilon) t\}\nonumber\\
&=&\overline{F}_1((1-\varepsilon )t)\overline{F}_2(\varepsilon t)+P\{X_1>t,X_2\le \varepsilon t\}
+P\{X_1+X_2>t,X_1\le t, X_2\le \varepsilon t\}\nonumber\\
&&+P\{X_1\le (1-\varepsilon) t,X_2>t\}+P\{X_1+X_2>t, X_1\le (1-\varepsilon )t,X_2\le t\}\nonumber\\
&=&\overline{F}_1((1-\varepsilon )t)\overline{F}_2(\varepsilon t)+\overline{F}_1(t)-\overline{F}_1(t)\overline{F}_2(\varepsilon t)
+\pi_1(t,\varepsilon)\nonumber\\
&&+\overline{F}_2(t)-\overline{F}_2(t)\overline{F}_1((1-\varepsilon )t)+\pi_2(t,\varepsilon),
\end{eqnarray}
where
\begin{eqnarray}
\pi_1(t,\varepsilon)&=&\int_0^{\varepsilon t}(\overline{F}_1(t-y)-\overline{F}_1(t))dF_2(y),\nonumber\\
\pi_2(t,\varepsilon)&=&\int_0^{(1-\varepsilon )t}(\overline{F}_2(t-y)-\overline{F}_2(t))dF_1(y).\nonumber
\end{eqnarray}

By the monotone density theorem (e.g., page 568 in Embrechts, Kluppelberg and Mikosch~\cite{Embrechts1997}), we know that
$f_1(t)\sim (d-1)c_1 t^{-d}L(t)\sim (d-1)t^{-1}\overline{F}_1(t)$.
By the mean value theorem, there exists $\theta=\theta(t,y)\in (0, 1)$ such that $\overline{F}_1(t-y)-\overline{F}_1(t)=f_1(t-\theta y)\cdot y$ for $0\le y\le \varepsilon t$. The decreasing property of $f_1$ yields $f_1(t)\le f_1(t-\theta y)\le f_1((1-\varepsilon )t)$ for $0\le y\le \varepsilon t$.
Therefore
\begin{eqnarray}
f_1(t)\int_0^{\varepsilon t}y dF_2(y)\le \pi_1(t,\varepsilon)\le f_1((1-\varepsilon )t)\int_0^{\varepsilon t}y dF_2(y),
\end{eqnarray}
which follows
\begin{eqnarray}\label{lemma-2nd-order-proof-supp2}
(d-1)\mu_{F_2}=\liminf_{t\to\infty}\frac {\pi_1(t,\varepsilon)} {t^{-1}\overline{F}_1(t)} \le\limsup_{t\to\infty}\frac {\pi_1(t,\varepsilon)} {t^{-1}\overline{F}_1(t)}=(d-1)(1-\varepsilon)^{-d}\mu_{F_2}
\end{eqnarray}

Next, we prove the following asymptotic result:
\begin{eqnarray}
\frac {\pi_2(t,\varepsilon)} {\overline{F}_2(t)}=\int_0^{(1-\varepsilon )t}\left[\frac{\overline{F}_2(t-y)} {\overline{F}_2(t)} -1\right]dF_1(y)&=& o(1).\label{lemma-2nd-order-proof-supp1}
\end{eqnarray}
Note that
\begin{eqnarray}\label{lemma-2nd-order-proof-supp3}
&&\int_0^{(1-\varepsilon )t}\left[\frac{\overline{F}_2(t-y)} {\overline{F}_2(t)} -1\right]dF_1(y)-\int_0^{(1-\varepsilon )t}\left[\left(1-{y\over t}\right)^{-d}-1\right]dF_1(y)\nonumber\\
&=&\int_0^{(1-\varepsilon )t}\left(1-{y\over t}\right)^{-d}\left[\frac {L(t(1-y/t))} {L(t)}(1+o(1))-1\right]dF_1(y)\nonumber\\
&=&o(1)
\end{eqnarray}
where the last equality is due to the the fact: by the uniform convergence theorem for slowly varying functions, $L(t(1-y/t))/L(t)\to 1$ uniformly on $y\in [0,(1-\varepsilon )t]$ (or $1-y/t\in [\varepsilon,1]$).

Since, for $b> 0$ and $0\le w\le 1-\varepsilon$,
\begin{eqnarray}
(1-w)^{-b}&=&1+bw+\frac {b(b+1)} {2!}w^2+\frac {b(b+1)(b+2)} {3!} w^3+\cdots\nonumber\\
&\le&1+bw\left[1+(b+1)w+\frac {(b+1)(b+2)} {2!} w^2+\cdots\right]\nonumber\\
&=&1+bw[(1-w)^{-b-1}]\le 1 + bw/\varepsilon^{b+1},
\end{eqnarray}
we have
\begin{eqnarray}\label{lemma-2nd-order-proof-supp4}
0\le\int_0^{(1-\varepsilon )t}\left[\left(1-{y\over t}\right)^{-d}-1\right]dF_1(y)\le \frac {d} {\varepsilon^{d+1}}\cdot\frac 1 t
\int_0^{(1-\varepsilon )t}ydF_1(y)
=o(1),
\end{eqnarray}
where the last equality is due to $\int_0^t x dF_1(x)=-t \overline{F}_1(t)+\int_0^t \overline{F}_1(x)dx$.
Now, (\ref{lemma-2nd-order-proof-supp1}) follows from (\ref{lemma-2nd-order-proof-supp3}) and (\ref{lemma-2nd-order-proof-supp4}).

Now, let us recall (\ref{lemma-2nd-order-proof-A}). Note that for all $0<\varepsilon<1$, $\overline{F}_1((1-\varepsilon )t)\overline{F}_2(\varepsilon t)=o(\overline{F}_2(t))$, $\overline{F}_1(t)\overline{F}_2(\varepsilon t)=o(\overline{F}_2(t))$, $\overline{F}_2(t)\overline{F}_1((1-\varepsilon )t)=o(\overline{F}_2(t))$ and $\pi_2(t,\varepsilon)=o(\overline{F}_2(t))$ (by \ref{lemma-2nd-order-proof-supp1})). Using (\ref{lemma-2nd-order-proof-supp2}) and taking $\varepsilon\to 0^+$, we complete the proof.
\end{proof}

\begin{remark}
In Lemma~\ref{lemma-F1F2}, if $X_1$ is a nonnegative integer valued r.v., then the condition that $f_1(x)$ is ultimately decreasing should be replaced by that $P\{X_1=j\}$ is ultimately decreasing, because we can define $f_1(t)\stackrel{\rm def}{=}P\{X_1=j\}$ for $j\le t<j+1$ and $j\ge 0$.
\end{remark}

\begin{remark} The ultimately decreasing condition imposed in Lemma~\ref{lemma-F1F2} is, in general, a non-trivial one to justify. We discuss this condition in Section~\ref{sec:8}.
\end{remark}

Recalling (\ref{L=L+D}), (\ref{Linfinity-6.7}) and (\ref{main-1}), applying Lemma~\ref{lemma-F1F2} with the setting of $X_1=L_{\infty}$, $X_2=D^{(0)}$ and $X_1+X_2=L_{\mu}$, and noting that $E(X_2)=\frac d {dz}D^{(0)}(z)\big |_{z=1}=\psi $ (see (\ref{D{(0)}(z)-2})), we conclude that
\begin{eqnarray}
    P\{L_{\mu}>j\}-P\{L_{\infty}>j\}\sim (a-1)\psi\cdot j^{-1}P\{L_{\infty}>j\} +P\{D^{(0)}>j\}.\label{pre-main-theorem}
\end{eqnarray}

The results (\ref{Linfinity-6.7}), (\ref{Lmu-Linfinity}), (\ref{pre-main-theorem}) and (\ref{main-1}) are  summarized in the following theorem.
\begin{theorem}[Main theorem -- a refined equivalence] \label{Theorem refined}
For the stable $M^X/G/1$ retrial queue with assumptions~A1 and A2, we have the following asymptotic properties. As $j\to\infty$,
\begin{flalign}
\begin{split}
P\{L_{\mu}>j\}\sim P\{L_{\infty}>j\}\sim c_{K^{\circ}}\cdot j^{-a+1}L(j), \label{refined-1}
\end{split}
\end{flalign}
furthermore, if $P\{L_{\infty}=j\}$ is ultimately decreasing in $j$, then
\begin{flalign}
\begin{split}
\quad P\{L_{\mu}>j\}-P\{L_{\infty}>j\}\sim \left[(a-1)\psi c_{K^{\circ}}+c_{D^{(0)}}\right]\cdot j^{-a}L(j),
\end{split}
\end{flalign}
where $\psi$, $c_{K^{\circ}}$ and $c_{D^{(0)}}$ are given in (\ref{psi}), (\ref{c-K-circ}) and (\ref{main-2}), respectively.
\end{theorem}
\begin{remark}
It is worth mentioning that in the first part of Theorem \ref{Theorem refined}, the asymptotic equivalence $P\{L_{\mu}>j\}\sim P\{L_{\infty}>j\}$
is proved without the assumption of a lighter tail for the batch size than that for the service time. In contrast, this equivalence was verified with the assumption of a light-tailed batch size in \cite{Yamamuro:2012} or a moderately heavy-tailed batch size in \cite{Masuyama:2014}, but in both the batch size has a tail lighter than that for the service time.
\end{remark}

\section{Asymptotic property for the tail probability of the rv $D^{(1)}$}
\label{sec:7}
Recall the definition of the rv $D^{(1)}$ in Section~\ref{sec:2}, i.e.,
$D^{(1)}$ is a rv having the distribution equal to the conditional distribution of the number of repeated customers in the orbit given that the server is busy.
Consider $D^{(1)}(z)\stackrel{\rm def}{=}E(z^{N_{orb}}|C_{ser}=1)$. Note that $P\{C_{ser}=1\}=\rho$.
The following result on $D^{(1)}(z)$ is from (Falin and Templeton~\cite{Falin-Templeton:1997}, pp.174):
\begin{eqnarray}
D^{(1)}(z)\stackrel{\rm def}{=}E(z^{N_{orb}}|C_{ser}=1)&=&\frac {1-\beta(\lambda-\lambda X(z))} {\beta(\lambda-\lambda X(z))-z}\cdot \frac {1-\rho} {\rho}\cdot D^{(0)}(z),\label{6-0}
\end{eqnarray}
where $D^{(0)}(z)$ is given in (\ref{D{(0)}(z)}). Rewritting (\ref{6-0}) gives
\begin{eqnarray}
D^{(1)}(z)&=&\frac {1-\beta(\lambda-\lambda X(z))} {(\lambda-\lambda X(z))\beta_1} \cdot\frac {1- X(z)} {(1-z)\chi_1 }\cdot \frac {(1-\rho)(1-z)} {\beta(\lambda-\lambda X(z))-z}\cdot D^{(0)}(z)\nonumber\\
&=&\beta^{(e)}(\lambda-\lambda X(z))\cdot X^{(de)}(z)\cdot K^{\circ}(z)\cdot D^{(0)}(z),\label{6-1}
\end{eqnarray}
where $K^{\circ}(\cdot)$ is defined in (\ref{K{circ}(u)}), $\beta^{(e)}(\lambda-\lambda X(z))\cdot X^{(de)}(z)$ is stated in Fact D.

It follows from (\ref{6-1}) that
\begin{eqnarray}\label{H''}
D^{(1)}& \stackrel{\rm def}{=} &N_{B^{(e)}_XX^{(de)}}+K^{\circ}+D^{(0)},
\end{eqnarray}
where $N_{B^{(e)}_XX^{(de)}}$, $K^{\circ}$ and $D^{(0)}$, stated in Sections \ref{sec:2} and \ref{decomposition}, are independent rvs having GFs $\beta^{(e)}(\lambda-\lambda X(z))\cdot X^{(de)}(z)$, $K^{\circ}(z)$ and $D^{(0)}(z)$, respectively.
It follows from (\ref{main-1}) and (\ref{K-circ-1}) that $P\{D^{(0)}>j\}=o(P\{K^{\circ}>j\})$, hence
\begin{eqnarray}
P\{D^{(1)}>j\}
&\sim&P\{N_{B^{(e)}_XX^{(de)}}+K^{\circ}>j\}.\label{6-3}
\end{eqnarray}

Similar to $P\{D^{(0)}>j\}$, our discussion on $P\{D^{(1)}>j\}$ is divided into three cases, which is essentially based on whether the batch size $X$ has a tail lighter than, heavier than, or equivalent to that for the service time $B$.

\noindent{\it Case 1. $d_X>d_B$ in Assumptions~A1 and A2:}

In this case, the asymptotic property for the tail probabilities of $P(N_{B^{(e)}_XX^{(de)}}>j)$ and $P\{K^{\circ}>j\}$ as $j\to\infty$, are given in
(\ref{NBeXXe-NBeX}) and (\ref{K^{(b)}>n}), respectively. Applying Part (ii) of Lemma \ref{Lemma 3.4new}, we get
\begin{eqnarray}
P\{D^{(1)}>j\}\sim \frac {(\lambda\chi_1)^{d_B}} {(d_B-1)(1-\rho)\rho}\cdot j^{-d_B+1}L(j),\quad j\to\infty.\label{D1-case1}
\end{eqnarray}

\noindent{\it Case 2. $d_X<d_B$ and $c_X> 0$ in Assumptions~A1 and A2:}

In this case, the asymptotic property for the tail probabilities of $P(N_{B^{(e)}_XX^{(de)}}>j)$ and $P\{K^{\circ}>j\}$ as $j\to\infty$, are given in
(\ref{NBeXXe-NBeX-case2}) and (\ref{K^{(b)}>n-case2}), respectively. Applying Lemma \ref{Lemma 3.4new}, we get
\begin{eqnarray}
P\{D^{(1)}>j\}\sim \frac {\lambda\beta_1c_X} {(d_X-1)(1-\rho)\rho}\cdot j^{-d_X+1}L(j),\quad j\to\infty.\label{D1-case2}
\end{eqnarray}

\noindent{\it Case 3. $d_X=d_B=a$ and  $c_X> 0$ in Assumptions~A1 and A2:}

In a manner similar to Cases 1 and 2, one can prove
\begin{eqnarray}
P\{D^{(1)}>j\}\sim \frac {(\lambda\chi_1)^{a}+\lambda\beta_1c_X} {(a-1)(1-\rho)\rho}\cdot j^{-a+1}L(j),\quad j\to\infty,\label{D1-case3}
\end{eqnarray}
where we have skipped the detailed derivations to avoid the repetition.

The above results in three cases are summarized in the following theorem.
\begin{theorem}\label{Theorem 5.1}
Under A1 and A2,
\begin{eqnarray}
P\{D^{(1)}>j\}&\sim& c_{D^{(1)}}\cdot j^{-a+1}L(j),\quad\mbox{as }j\to\infty,\label{Theorem3-2a}
\end{eqnarray}
where $a=\min(d_B, d_X)>1$ and
\begin{eqnarray}
c_{D^{(1)}}=\left\{
\begin{array}{ll}
(\lambda\chi_1)^{a}/((a-1)(1-\rho)\rho), &\quad \mbox{ if }d_X>d_B>1,\\
\lambda\beta_1c_X/((a-1)(1-\rho)\rho), &\quad \mbox{ if }1<d_X<d_B \mbox{ and}\ c_X>0, \\
((\lambda\chi_1)^{a}+\lambda\beta_1c_X)/((a-1)(1-\rho)\rho), &\quad \mbox{ if }d_X=d_B>1 \mbox{ and}\ c_X>0.
\end{array}
\right.
\end{eqnarray}
\end{theorem}

\section{Discussions on the key condition in Lemma~\ref{lemma-F1F2}}
\label{sec:8}

The ultimately decreasing property of $P\{L_\infty =j\}$ is the key condition for applying Lemma~\ref{lemma-F1F2}. This type of condition could appear in various situations in queueing analysis. For example, it is a standard assumption in many Tauberian theorems. In general, verification of this condition can be non-trivial, since the probability distribution itself is unknown. In this section, we propose a procedure for verifying the ultimately decreasing condition for $P\{L_\infty =j\}$ for the two possible cases: in the first case, the tail probability of the service time is heavier than that of the batch size, and the second case is vice versa. We demonstrate this procedure in detail in terms of two specific models.  This procedure is based on the asymptotic expansion of $Ez^{L_{\infty}}$, which is valid if the models allow for the detailed second term in the expansion.

Before proceeding to the examples, we prove the following lemma, and present a few literature properties for expansions, which will be used later.

\begin{lemma}\label{exmp-lemma-1}
Suppose that  $x_n=c_0n^{-d_0}+\sum_{i=1}^{k} c_i n^{-d_i} +o(n^{-d_0-1})$ for $n\ge 0$, where $c_0>0$, $d_0>0$, $k\ge 0$, $d_0<d_i\le d_0+1$ and
$-\infty<c_i<\infty$ for $i=1,\cdots,k$. Then, $\{x_n\}_{n=0}^{\infty}$ is an ultimately decreasing sequence.
\end{lemma}

\begin{proof}
Let $\Delta x_n=x_{n+1}-x_n$ for $n\ge 0$. Then,
\begin{eqnarray}
\Delta x_n&=&c_0 n^{-d_0}\big[(1+1/n)^{-d_0}-1\big]+\sum_{i=1}^{k} c_i n^{-d_i}\big[(1+1/n)^{-d_i}-1\big]+o(n^{-d_0-1})\nonumber\\
&=&-c_0 d_0 n^{-d_0 -1} +o(n^{-d_0-1}),
\end{eqnarray}
which is ultimately negative.
\end{proof}

It is worthwhile to mention that identifying $c_0$ and $d_0$, which satisfies the condition of Lemma~\ref{exmp-lemma-1}, is the key for verifying the ultimately decreasing condition of $x_n$, while specific values for $c_i$ and $d_i$ for $i >0$ are not important.

Recalling (6.4)\ref{eqn:6.4} and (\ref{K{circ}(z)-2}), we have
\begin{eqnarray}\label{without-retrial-1}
Ez^{L_{\infty}}&=&\beta(\lambda-\lambda X(z))\cdot K^{\circ}(z),\label{EzL{infty}}\\
K^{\circ}(z)&=&\frac {1-\rho}{1-\rho\beta^{(e)}(\lambda-\lambda X(z))\cdot X^{(de)}(z)}.\label{exmp-Kcirc}
\end{eqnarray}

Remember that $P\{L_{\infty}=j\}$ is given as the coefficient of the term $z^j$ in the generating function $Ez^{L_{\infty}}$. To study the tail asymptotic property of the coefficient sequence, it is convenient to establish the correspondence between the asymptotic property of the generating function at its dominant singularity ($z=1$ in our case) and the tail asymptotic property in the coefficient sequence. For this purpose, the following facts are useful. Readers may refer to pages 381--392, Theorem~VI.1 and Theorem~VI.3(ii) in Flajolet and Sedgewick~\cite{Flajolet-Sedgewick:2008} for details of facts~(i), (ii) and (iii), respectively. For a definition of the $\Delta$-analyticity, see, for example, Definition~VI.1 in \cite{Flajolet-Sedgewick:2008}.

\begin{lemma}[Flajolet and Sedgewick~\cite{Flajolet-Sedgewick:2008}] \label{lemma:FS}
\begin{description}
\item[(i)] Let $\mathcal{P}(s)$ be any polynomial of degree $k\ge 0$, i.e., $\mathcal{P}(s)=p_0+p_1 s +p_2 s^2 +\cdots +p_{k}s^{k}$. Assume that $F(z)=\sum_{n=0}^{\infty}f_n z^{n}=\mathcal{P}(1-z)$. Then, there exists some $n_0$ such that $f_n\equiv 0$ for $n\ge n_0$. A special case is $F(z)=(1-z)^{m}$, where $m$ is a non-negative integer.

\item[(ii)] Assume that $F(z)=\sum_{n=0}^{\infty}f_n z^{n}=(1-z)^{\sigma}$, where $\sigma\neq 0,1,2,\cdots$. Then, $f_n=n^{-\sigma-1}/\Gamma(-\sigma)\Big(1+\frac{\sigma(\sigma+1)} {2}n^{-1} +o(n^{-1})\Big)$ as $n\to\infty$.

\item[(iii)] Assume that $F(z)=\sum_{n=0}^{\infty}f_n z^{n}$ is $\Delta$-analytic at $1$, and $F(z)=o((1-z)^{\sigma})$ as $z$ goes to $1$ within  the $\Delta$-domain at $1$, where $\sigma\neq 0,1,2,\cdots$. Then, $f_n=o(n^{-\sigma-1})$ as $n\to\infty$.
\end{description}
\end{lemma}

By property~(i) of the above lemma, we do not need to specify the expression for any polynomial $\mathcal{P}(1-z)$ in our later asymptotic analysis.

\subsection{Example A: Pareto service time and geometric batch size}

In this subsection, we demonstrate the procedure by an example for the case that the tail probability of the service time is heavier than the tail probability of the batch size. Specifically, assume
that the service time $B$ has a Pareto distribution with {\it non-integer} shape parameter  $a>1$ and scale parameter $b>0$, i.e.,
$B(x)=1-\left(1+ x/b\right)^{-a}$ for $x\ge 0$, and $\beta_1=E(B)=b/(a-1)$.
Suppose that the batch size $X$ has a geometric distribution $P\{X=n\}=(1-q) q^{n-1}$, $0< q<1$, $n=1,2,\ldots$, with mean $\chi_1=E(X)= 1/(1-q)$ and the GF
$X(z)= (1-q)z/(1-qz)$. It follows from the definition (see Fact D) that $X^{(de)}$ has the GF $X^{(de)}(z)= (1-q)/(1-qz)$.

Let $\mathbb{C}$ denote the complex plane and $\mathbb{D}=\mathbb{C}\backslash(-\infty,0]$ denote the complex plane $\mathbb{C}$ with the negative real axis $(-\infty,0]$ (the branch cut) removed. By the result of Goovaerts et al. (see (3.15) in \cite{Goovaerts-D'Hooge-Pril:1977}), the LSTs $\beta(s)$ and $\beta^{(e)}(s)$ can be defined by analytical continuation for $s\in \mathbb{D}$ as follows:
\begin{eqnarray}
\beta(s)&=&1+\sum_{n=1}^{\infty}\frac {(-bs)^n}{(a-1)\cdots (a-n)} -\Gamma (1-a)(bs)^a e^{bs},\label{exmp-beta(s)-exact}\\
\beta^{(e)}(s)=\frac {1-\beta(s)} {\beta_1 s} &=&1+\sum_{n=1}^{\infty}\frac {(-bs)^n}{(a-2)\cdots (a-1-n)} -\Gamma (2-a)(bs)^{a-1} e^{bs}.\label{exmp-beta(e)(s)-exact}
\end{eqnarray}

It can be proved that $Ez^{L_{\infty}}$, given in (\ref{EzL{infty}}) and (\ref{exmp-Kcirc}), can be analytically continued to $\mathbb{C}\backslash [1,\infty)$ (see Appendix~\ref{app-B} for details).
Using the fact that $e^{bs} =1 + bs +o(s)$ as $s\to 0$, we have
\begin{eqnarray}
\beta(s)&=&1+s\mathcal{R}(s) -\Gamma (1-a) (bs)^a + o(s^{a}),\label{exmp-beta(s)-asym}\\
\beta^{(e)}(s)&=&1+s\mathcal{Q}(s) -\Gamma (2-a)\big[(bs)^{a-1}+(bs)^{a}\big] + o(s^{a}),\label{exmp-beta(e)(s)-asym}
\end{eqnarray}
where $\mathcal{R}(s)$ and $\mathcal{Q}(s)$ are two polynomials of degree $\lfloor a\rfloor-1$ (the symbol $\lfloor a\rfloor$ denotes the integral part of $a$).
In addition,
\begin{eqnarray}
1-X(z)&=&\frac {(1-z)/(1-q)} {1+q(1-z)/(1-q)}\ =\ \frac {1-z} {1-q}\left[1 + \sum_{k=1}^{\infty}\left(-q\cdot\frac {1-z} {1-q} \right)^k\right],\label{exmp-X(z)-asym}\\
\big(1-X(z)\big)^a&=&\left(\frac {1-z} {1-q}\right)^a +o((1-z)^{a}),\quad \mbox{ as }z\uparrow 1,\label{exmp(1-X(z))^(a)-asym}\\
\big(1-X(z)\big)^{a-1}&=&\left(\frac {1-z} {1-q}\right)^{a-1} - (a-1)q \left(\frac {1-z} {1-q}\right)^{a} +o((1-z)^{a}),\quad \mbox{ as }z\uparrow 1,\label{exmp(1-X(z))^(a-1)-asym}
\end{eqnarray}
where the fact that $(1-x)^{b}=1-bx +o(x)$ as $x\to 0$ was used.

By (\ref{exmp-beta(e)(s)-asym}), (\ref{exmp(1-X(z))^(a)-asym}) and (\ref{exmp(1-X(z))^(a-1)-asym}), we have
\begin{eqnarray}
\beta^{(e)}(\lambda-\lambda X(z))&=&1+(1-z)\mathcal{Q}_1(1-z) -\Gamma (2-a)
\left(b\lambda\cdot\frac {1-z} {1-q}\right)^{a-1} \nonumber\\
&&+h_1\cdot(1-z)^{a}+ o((1-z)^{a}),\quad \mbox{ as }z\uparrow 1,\label{exmp-beta(e)(lambda-lambdaX(z))}
\end{eqnarray}
where $\mathcal{Q}_1(s)$ is a polynomial of degree $\lfloor a\rfloor-1$, and $h_1$ is a real number.

By (\ref{exmp-X(z)-asym}),
\begin{eqnarray}
X^{(de)}(z)&=&1 + \sum_{k=1}^{\infty}\left(-q\cdot\frac {1-z} {1-q} \right)^k.\label{exmp-X(de)(z)-asym}
\end{eqnarray}
It follows from (\ref{exmp-beta(e)(lambda-lambdaX(z))}) and (\ref{exmp-X(de)(z)-asym}) that
\begin{eqnarray}
\beta^{(e)}(\lambda-\lambda X(z))X^{(de)}(z)&=&1+(1-z)\mathcal{Q}_2(1-z) -\Gamma (2-a)
\left(b\lambda\cdot\frac {1-z} {1-q}\right)^{a-1}  \nonumber\\
&&+h_2\cdot(1-z)^{a} + o((1-z)^{a}),\quad \mbox{ as }z\uparrow 1,\label{exmp-beta(e)-X(de)}
\end{eqnarray}
where $\mathcal{Q}_2(s)$ is a polynomial of degree $\lfloor a\rfloor-1$, and $h_2$ is a real number.

Substituting (\ref{exmp-beta(e)-X(de)}) into (\ref{exmp-Kcirc}), we get, as $z\uparrow 1$,
\begin{eqnarray}
K^{\circ}(z)=
\frac {1} {1-(1-z)\mathcal{Q}_3(1-z) +\frac {\rho} {1-\rho}\Gamma (2-a)\left(b\lambda\cdot\frac {1-z} {1-q}\right)^{a-1} - h_3\cdot(1-z)^{a} + o((1-z)^{a})},\label{exmp-Kcirc-22}
\end{eqnarray}
where $\mathcal{Q}_3(s)$ is a polynomial of degree $\lfloor a\rfloor-1$, and $h_3$ is a real number.

Applying $1/(1-x) =1 + x  + x^2 +\cdots$ to (\ref{exmp-Kcirc-22}), we obtain
\begin{eqnarray}
K^{\circ}(z)&=&1+(1-z)\mathcal{Q}_4(1-z) -\frac {\rho} {1-\rho}\Gamma (2-a)\left(b\lambda\cdot\frac {1-z} {1-q}\right)^{a-1} + h_4\cdot(1-z)^{a} \nonumber\\
&& +\sum_{k\in S} g_{k}\cdot(1-z)^{k(a-1)} + o((1-z)^{a}),\quad \mbox{ as }z\uparrow 1,\label{exmp-Kcirc-33}
\end{eqnarray}
where $\mathcal{Q}_4(s)$ is a polynomial of degree $\lfloor a\rfloor-1$, $S=\{k\ |\ k\ge 2,\ k(a-1)\le a\}$, $h_4$ and $g_k$'s are real numbers.

Similarly, by (\ref{exmp-beta(s)-asym}) and (\ref{exmp(1-X(z))^(a)-asym}), we have
\begin{eqnarray}
\beta(\lambda-\lambda X(z))&=&1 + \mathcal{R}_1(1-z) +h_5\cdot (1-z)^{a}
+ o((1-z)^{a}),\quad \mbox{ as }z\uparrow 1.\label{exmp-beta(lambda-lambda X(z))-44}
\end{eqnarray}
where $\mathcal{R}_1(s)$ is a polynomial of degree $\lfloor a\rfloor-1$, and $h_5$ is a real number.

It follows from (\ref{EzL{infty}}), (\ref{exmp-Kcirc-33}) and (\ref{exmp-beta(lambda-lambda X(z))-44}) that
\begin{eqnarray}
Ez^{L_{\infty}}&=&1+\mathcal{R}_2(1-z) -\frac {\rho} {1-\rho}\Gamma (2-a)\left(b\lambda\cdot\frac {1-z} {1-q}\right)^{a-1} + r_2\cdot(1-z)^{a} \nonumber\\
&& +\sum_{k\in S} g_{k}\cdot(1-z)^{k(a-1)} + o((1-z)^{a}),\quad \mbox{ as }z\uparrow 1,\label{examp-EzLinfty-21}
\end{eqnarray}
where $\mathcal{R}_2(s)$ is a polynomial of degree $\lfloor a\rfloor-1$, and $r_2$ is a real number.

Since $Ez^{L_{\infty}}$ is analytic in $\mathbb{C}\backslash [1,\infty)$, which implies the $\Delta$-analyticity of $Ez^{L_{\infty}}$ at $1$.
Applying properties~(i)--(iii) of Lemma~\ref{lemma:FS} to (\ref{examp-EzLinfty-21}), we obtain the following result.
\begin{lemma}
For the $M^X/G/1$ model with the Pareto service time and geometric batch size, we have,
as $j\to\infty$,
\begin{eqnarray}
P\{L_{\infty}=j\}&=&\frac {1} {1-\rho}\left(\frac {b\lambda} {1-q}\right)^{a} j^{-a}+ l_1 j^{-a-1} +\sum_{k\in S} l_{k}j^{-k(a-1)-1} + o(j^{-a-1}),\label{exmp-PLinfty=j-last}
\end{eqnarray}
where we have used the fact $-\frac {\rho} {1-\rho}\frac{\Gamma (2-a)} {\Gamma(1-a)} \big(\frac {b\lambda} {1-q}\big)^{a-1}=
\frac {1} {1-\rho}\big(\frac {b\lambda} {1-q}\big)^{a}$,
and $l_k$'s are real numbers.

Furthermore, $P\{L_{\infty}=j\}$ is ultimately decreasing, which is a directly consequence of applying Lemma~\ref{exmp-lemma-1} to (\ref{exmp-PLinfty=j-last}).
\end{lemma}

\begin{remark}
(i) Let $q\to 0$. The batch size $X$ has a degenerative distribution $P\{X= 1\}=1$. So, for the $M/G/1$ queue with Pareto service, $P\{L_{\infty}=j\}$ is ultimately decreasing. This result can also be directly verified from an exact formula for $P\{L_{\infty}=j\}$ (see (29) and (39) in \cite{Ramsay:2007}).
(ii) For the $M/G/1$ queue with Pareto service, $P\{L_{\infty}=j\}\sim\frac {(b\lambda)^{a} } {1-\rho} j^{-a}$ (by (\ref{exmp-PLinfty=j-last}) with $q= 0$), it follows that $P\{L_{\infty}>j\}\sim\frac {(b\lambda)^{a}} {(1-\rho)(a-1)} j^{-a+1}=\frac {\rho} {1-\rho} (j/b\lambda)^{-a+1}\sim \frac {1} {1-\rho}P\{B^{(e)}>j/\lambda\}$ as $j\to\infty$, which is consistent with (1.4) in \cite{Asmussen-Klupperlberg-Sigman:1999}.
\end{remark}


\subsection{Example B: Pareto service time and batch size heavier than service time}

In this subsection, we apply the same procedure, used in the previous subsection, to another example for the case that the tail of the batch size distribution is heavier than that of the service time distribution. Specifically, we assume that
the service time $B$ has a Pareto distribution with shape parameter  $1<a<2$ and scale parameter $b>0$, i.e.,
$B(x)=1-\left(1+ x/b\right)^{-a}$ for $x\ge 0$, and $\beta_1=E(B)=b/(a-1)$. For the batch size $X=X_0+1$, we assume that
$X_0(z)=E(z^{X_0})=\tau(1-z)$, where $\tau(s)$ is the LST of the continuous-time distribution of another Pareto r.v. $T$ with $P\{T\le t\}=1-\left(1+ t/v\right)^{-u}$ for $t>0$, $v>0$ and $1<u<2$. So, $\tau_1= E(T)= v/(u-1)$. In addition, we assume that  $u<a$.
Clearly, $X_0$ can be regarded as the number of Poisson arrivals at rate $1$ within a random time $T$. Since $P\{X> j\}\sim P\{X_0> j\}\sim P\{T> j\}$ as $j\to\infty$ (see Lemma A.1), the condition $u<a$ implies that $X$ has a tail heavier than $B$. Moreover, $E(X_0)= E(T)=\tau_1$, so $\chi_1= E(X)=1+ \tau_1$.

Once again, it can be proved that $Ez^{L_{\infty}}$ given in (\ref{EzL{infty}}) and (\ref{exmp-Kcirc}) can be analytically continued to $\mathbb{C}\backslash [1,\infty)$, which implies the $\Delta$-analyticity at 1 (a detailed proof is given in Appendix~\ref{app-B}). Therefore, the results in Lemma~\ref{lemma:FS} can be employed for this example.

In a way similar to (\ref{exmp-beta(s)-exact}) and (\ref{exmp-beta(s)-asym}), we have, for $s\in \mathbb{D}$, as $s\to 0$,
\begin{eqnarray}
\tau(s)&=&1- \tau_1 s -w_1 s^2-\Gamma (1-u)\big[(vs)^u + (vs)^{u+1}\big]+ o(s^{u+1}),\label{exmp2-delta(s)-asym}
\end{eqnarray}
where $w_1$ is a real number.
Note that $1-X(z)=1-z\tau(1-z)=1-\tau(1-z)+(1-z)\tau(1-z)$. Therefore, as $z\uparrow 1$,
\begin{equation}
    1-X(z) = \chi_1(1-z)\big[1+w_2(1-z) +(v^u/\chi_1)\Gamma (1-u) (1-z)^{u-1} +w_3(1-z)^{u}+o((1-z)^{u})\big],\label{exmp2-alpha(s)-asym}
\end{equation}
where $w_2$ and $w_3$ are real numbers. It follows that
\begin{eqnarray}
\big(1-X(z)\big)^a&=&o((1-z)^{u}),\label{exmp2(1-X(z))^(a)-asym}\\
\big(1-X(z)\big)^{a-1}&=&\sum_{k\in S}  g_{k}\cdot(1-z)^{a-1+k(u-1)} +o((1-z)^{u}),\label{exmp2(1-X(z))^(a-1)-asym}
\end{eqnarray}
where $S=\{k\ |\ k\ge 0,\ a-1+k(u-1)\le u\}$, and $g_k$'s are real numbers.

Substituting (\ref{exmp2-alpha(s)-asym}), (\ref{exmp2(1-X(z))^(a)-asym}) and (\ref{exmp2(1-X(z))^(a-1)-asym}) into (\ref{exmp-beta(e)(s)-exact}) leads to
\begin{eqnarray}
\beta^{(e)}(\lambda-\lambda X(z))&=&1 + w_4(1-z) +\sum_{k\in S_1} e_{k}\cdot(1-z)^{a-1+k(u-1)} +o((1-z)^{u}),\label{exmp2-beta(e)(lambda-lambdaX(z))}
\end{eqnarray}
where $w_4$ and $e_k$'s are real numbers.
In addition, by (\ref{exmp2-alpha(s)-asym}), we have
\begin{eqnarray}
X^{(de)}(z)=1+w_2(1-z) +(v^u/\chi_1)\Gamma (1-u) (1-z)^{u-1}
+ w_3(1-z)^{u}+ o((1-z)^{u}).\label{exmp2-X(de)(z)-asym}
\end{eqnarray}
It follows from (\ref{exmp2-beta(e)(lambda-lambdaX(z))}) and (\ref{exmp2-X(de)(z)-asym}) that
\begin{eqnarray}
\beta^{(e)}(\lambda-\lambda X(z))X^{(de)}(z)&=&1+w_5(1-z) +(v^u/\chi_1)\Gamma (1-u) (1-z)^{u-1}+ w_3(1-z)^{u}\nonumber\\
&&+\sum_{k\in S_1} r_{k}\cdot(1-z)^{a-1+k(u-1)} + o((1-z)^{u}),\label{exmp2-beta(e)-X(de)}
\end{eqnarray}
where $w_5$ and $r_k$'s are real numbers.
In a way similar to obtaining (\ref{exmp-Kcirc-33}), we get
\begin{eqnarray}
K^{\circ}(z)&=&1+w_6(1-z) +\frac {\rho} {1-\rho}(v^u/\chi_1)\Gamma (1-u) (1-z)^{u-1}+ w_7(1-z)^{u}\nonumber\\
&&+\sum_{(k_1,k_2)\in S_2}g_{k_1,k_2}\cdot(1-z)^{k_1(a-1)+k_2(u-1)}
+\sum_{k\in S_3} h_{k}\cdot(1-z)^{k(u-1)}\nonumber\\
&&+ o((1-z)^{u}),\quad \mbox{ as }z\uparrow 1,\label{exmp2-Kcirc-33}
\end{eqnarray}
where $S_2=\{(k_1,k_2)\ |\ k_1\ge 1,\ k_2\ge 0,\ k_1(a-1)+k_2(u-1)<u\}$, $S_3=\{k\ |\ k\ge 2,\ k(u-1)<u\}$ and
$g_{k_1,k_2}$'s and $h_k$'s are real numbers.

Substitute (\ref{exmp2(1-X(z))^(a)-asym}) into (\ref{exmp-beta(s)-asym}) to give
\begin{eqnarray}
\beta(\lambda-\lambda X(z))&=&1 -\beta_1(1-z) + o((1-z)^{u}),\quad \mbox{ as }z\uparrow 1.\label{exmp2-beta(lambda-lambda X(z))-44}
\end{eqnarray}
It follows from (\ref{EzL{infty}}), (\ref{exmp2-Kcirc-33}) and (\ref{exmp2-beta(lambda-lambda X(z))-44}) that
\begin{eqnarray}
Ez^{L_{\infty}}&=&1+w_8(1-z) +\frac {\rho} {1-\rho}(v^u/\chi_1)\Gamma (1-u) (1-z)^{u-1}+ w_9(1-z)^{u}\nonumber\\
&&+\sum_{(k_1,k_2)\in S_2}g^*_{k_1,k_2}\cdot(1-z)^{k_1(a-1)+k_2(u-1)}
+\sum_{k\in S_3} h^*_{k}\cdot(1-z)^{k(u-1)}\nonumber\\
&&+ o((1-z)^{u}),\quad \mbox{ as }z\uparrow 1,\label{exmp2-Ezinfty-33}
\end{eqnarray}
where $w_8$, $w_9$, $g^*_{k_1,k_2}$'s and $h^*_k$'s are real numbers.

Since $P\{L_{\infty}=j\}$ is analytic in $\mathbb{C}\backslash [1,\infty)$, which implies the $\Delta$-analyticity at 1, we are able to
apply Lemma~\ref{lemma:FS} to (\ref{exmp2-Ezinfty-33}) to obtain the following results.
\begin{lemma} For the $M^X/G/1$ queue with Pareto batch size and Pareto service time, if the tail of the batch size is heavier than the tail of the service time, we have
\begin{eqnarray}
P\{L_{\infty}=j\}&=&\frac {\rho} {1-\rho}(v^u/\chi_1)\Gamma (1-u) j^{-u}+ w_{10}j^{-u-1}
+\sum_{(k_1,k_2)\in S_2} l_{k_1,k_2}\cdot j^{-k_1(a-1)-k_2(u-1)-1}\nonumber\\
&&+\sum_{k\in S_3}l_{k}\cdot j^{-k(u-1)-1} + o(j^{-u-1}),\quad \mbox{ as }j\to\infty,\label{exmp2-PLinfty=j-last}
\end{eqnarray}
where we have used the fact
$\frac {\rho} {1-\rho}\frac {v^u} {\chi_1} j^{-u}=\frac {\lambda\beta_1} {1-\rho}  (j/v)^{-u}$,
and $w_{10}$, $l_{k_1,k_2}$'s and $l_k$'s are real numbers.

Furthermore, $P\{L_{\infty}=j\}$ is ultimately decreasing, which is a directly consequence of applying Lemma~\ref{exmp-lemma-1} to (\ref{exmp2-PLinfty=j-last}).
\end{lemma}

\section*{Acknowledgments}
We acknowledge the valuable comments/suggestions made by anonymous reviewers and the Associate Editor, which have led to this final version with significantly improved quality. This work was supported in part by the National Natural Science Foundation of China (Grant No. 71571002),
and a Discovery Grant from the Natural Sciences and Engineering Research Council of Canada (NSERC).

\appendix

\section{Collection of concepts and results}
\label{app-A}

\begin{definition}[e.g., see Bingham et al.~\cite{Bingham:1989}]
\label{Definition 3.1}
A measurable function $U:(0,\infty)\to (0,\infty)$ is regularly varying at $\infty$ with index $\sigma\in(-\infty,\infty)$, denoted by $U\in R_{\sigma}$, iff $\lim_{x\to\infty}U(tx)/U(x)=t^{\sigma}$ for all $t>0$. If $\sigma=0$ we call $U$ slowly varying, i.e., $\lim_{x\to\infty}U(tx)/U(x)=1$ for all $t>0$.
\end{definition}
%
\begin{definition}[e.g., see Foss et al.~\cite{Foss2011}]
\label{Definition 3.2}
A distribution $F$ on $(0,\infty)$ belongs to the class of the subexponential distributions, denoted by $F\in \mathcal S$, if $\lim_{x\to\infty}(1-F^{(2)}(x))/(1-F(x))=2$, where $F^{(2)}$ denotes the second convolution of $F$.
\end{definition}
%

\begin{lemma}[Asmussen et al.~\cite{Asmussen-Klupperlberg-Sigman:1999}, or Foss and Korshunov~\cite{Foss-Korshunov:2000}]
\label{Lemma 3.1}
Assume that $N_t$ is a Poisson process with rate $\lambda>0$, and $T>0$ is a rv
independent of $N_t$ with tail $P\{T>x\}$ heavier than $e^{-\sqrt{x}}$. Then $P(N_T>j)\sim P\{T>j/\lambda\},\ j\to\infty$.
\end{lemma}

Note that by Assumption A1, both the service time $B$ and the equilibrium service time $B^{(e)}$ have tails heavier than $e^{-\sqrt{x}}$.

\begin{lemma}
\label{Lemma 3.2}
Let $N$ be a discrete non-negative integer-valued rv,
and let $\{Y_k\}_{k=1}^{\infty}$ be a sequence of non-negative, independently and identically distributed rvs.
Define $S_0\equiv 0$ and $S_n=\sum_{k=1}^n Y_k$.

\begin{description}
\item[(i)] If $P\{Y_k>x\}\sim c_Y x^{-h}L(x)$ as $x\to\infty$ and $P\{N>n\}\sim c_N n^{-h}L(n)$ as $n\to\infty$, where $h> 1$, $c_Y\ge 0$ and $c_N\ge 0$, then
\begin{eqnarray}
P\{S_N > x\}&\sim&\left(c_N \mu_Y^{h}+  c_Y\mu_N\right) x^{-h}L(x),\quad x\to \infty,
\end{eqnarray}
where $E(N)=\mu_N<\infty$ and $E(Y_k)=\mu_Y<\infty$.

\item[(ii)] If $P\{N>n\}\sim c_N n^{-h_N}L(n)$ as $n\to\infty$, where $0\le h_N<1,\ c_N\ge 0$, and $E(Y_k)=\mu_Y<\infty$, then
\begin{eqnarray}
P\{S_N > x\}&\sim& c_N (x/\mu_Y)^{-h_N}L(x),\quad x\to \infty.
\end{eqnarray}


\item[(iii)] If $P\{N>n\}\sim c_N n^{-1}L(n)$ as $n\to\infty$, where $c_N\ge 0$, and $x^b P\{Y_k>x\}\le c<\infty$ for some $b>1$, then
\begin{eqnarray}
P\{S_N > x\}&\sim& c_N (x/\mu_Y)^{-1}L(x),\quad x\to \infty,
\end{eqnarray}
where $E(Y_k)=\mu_Y<\infty$.
\end{description}
\end{lemma}

In the above Lemma~\ref{Lemma 3.2}, Parts (i) and (ii) are directly from Corollary 8.1 and Corollary 8.2  in Grandell\cite{Grandell1997}, pp. 163-165, and Part (iii) is due to Lemma 2.8 in Stam~\cite{Stam:1973}, p. 315.

Lemma~\ref{Lemma 3.3} given below is the discrete version of Karamata's Theorem and Monotone Density Theorem.
\begin{lemma}[Embrechts et al.~\cite{Embrechts1997}, pp. 567-568]
\label{Lemma 3.3}
Let $\{q(j)\}_{j=0}^{\infty}$ be a nonnegative sequence, and $b>1$. If $q(j)\sim j^{-b}L(j)$ as $j\to\infty$, then
$\sum_{k=j+1}^{\infty}q(k)\sim \displaystyle\frac 1 {b-1} j^{-b+1}L(j)$ as $j\to\infty$. Conversely, if $\sum_{k=j+1}^{\infty}q(k)\sim\displaystyle\frac 1 {b-1} j^{-b+1}L(j)$ as $j\to\infty$
and $\{q(j)\}_{j=0}^{\infty}$ is ultimately monotonic, then $q(j)\sim j^{-b}L(j)$ as $j\to\infty$.
\end{lemma}

In the following lemma, the symbol ``$F_1*F_2$" stands for the convolution of $F_1$ and $F_2$.
\begin{lemma}[Foss et al.~\cite{Foss2011}, p.48]
\label{Lemma 3.4new}
Suppose that $F(x)\in\mathcal S$.
\begin{description}
\item[(i)] If $1-G(x)=o(1-F(x))$ as $x\to\infty$, then $F*G\in\mathcal S$ and $1-F*G(x)\sim 1-F(x)$.
\item[(ii)] If $(1-G_i(x))/(1-F(x))\to c_i$ as $x\to\infty$ for some $c_i\ge 0$, i=1,2, then
$(1-G_1*G_2(x))/(1-F(x))\to c_1+c_2$ as $x\to\infty$.
\end{description}
\end{lemma}

For proving our key result, Theorem~\ref{Theorem 3.2}, we need the following concepts and properties.
Let $\{g(j)\}_{j=0}^{\infty}$ be a discrete probability distribution with the GF $G(z)=\sum_{j=0}^{\infty}g(j)z^j$.
Denote by $\gamma_n(n\ge 0)$ the $n$th factorial moment of $\{g(j)\}_{j=0}^{\infty}$, this is,
\[
\gamma_0=1\quad\mbox{and}\quad\gamma_n=\sum_{k=n}^{\infty}k(k-1)\cdots (k-n+1)g(k),\quad n\ge 1.
\]
It is well known that if $\gamma_n<\infty$, then $\gamma_n=\lim_{z\uparrow 1}d^n G(z)/dz^n$ and
\begin{eqnarray}
G(z)=\sum_{k=0}^{n}(-1)^k\frac{\gamma_k}{k!}(1-z)^k +o((1-z)^n)\quad\mbox{ as }\ z\uparrow 1.
\end{eqnarray}
Next, if $\gamma_n<\infty$, we introduce notations $G_n(\cdot)$ and $\widehat{G}_n(\cdot)$ as follows:
\begin{eqnarray}
G_n(z)&\stackrel{\rm def}{=}&(-1)^{n+1}\left(G(z)-\sum_{k=0}^{n}(-1)^k\frac{\gamma_k}{k!}(1-z)^k\right),\quad n\ge 0,\label{G_n(z)}\\
\widehat{G}_n (z)&\stackrel{\rm def}{=}&\frac{G_n(z)}{(1-z)^{n+1}},\quad n\ge 0.\label{widehat{G}_n (z)}
\end{eqnarray}
So,
\begin{eqnarray}
G(z)=\sum_{k=0}^{n}(-1)^k\frac{\gamma_k}{k!}(1-z)^k+(-1)^{n+1}G_n(z).
\end{eqnarray}
It follows that if $\gamma_n<\infty$, then for $n\ge 1$,
\begin{eqnarray}
G_{n-1}(z)&=&\frac {\gamma_{n}} {n!}(1-z)^n-G_n(z),\\
\widehat{G}_{n-1}(z)&=&\frac {\gamma_{n}} {n!}-(1-z)\widehat{G}_n (z),\label{widehat{G}_{n-1}(z)}\\
\widehat{G}_{n-1}(1)&=&\frac {\gamma_n} {n!}-\lim_{z\uparrow 1}\frac {G_{n}(z)} {(1-z)^n}=\frac {\gamma_n} {n!}.\label{widehat{G}_{n-1}(1)}
\end{eqnarray}

In the following Lemma, we verify that $\widehat{G}_n (z)$ is the GF of a nonnegative sequence.
To this end, we define recursively
\begin{eqnarray}
\overline{g}_{0}(j)&=&g(j),\quad j\ge 0,\label{overline{g}_{0}(j)}\\
\overline{g}_{n+1}(j)&=&\sum_{i=j+1}^{\infty}\overline{g}_{n}(i),\quad j\ge 0;\ n\ge 0.\label{overline{g}_{n+1}(j)}
\end{eqnarray}

\begin{lemma}\label{Lemma 4.1}
Suppose that $\{g(j)\}_{j=0}^{\infty}$ is a discrete probability distribution with
$\gamma_{n}<\infty,\ n\ge 0$. Then $\widehat{G}_k(z)$ is the GF of sequence
$\{\overline{g}_{k+1}(j)\}_{j=0}^{\infty}$ for $0\le k\le n$, that is,
\begin{eqnarray}
\sum_{j=0}^{\infty}\overline{g}_{k+1}(j)z^j&=&\widehat{G}_k(z),\quad 0\le k\le n.\label{Lemma 4.1formula}
\end{eqnarray}
\end{lemma}

\begin{proof}
Notice that
\begin{eqnarray}
\sum_{j=0}^{\infty}\overline{g}_{k+1}(j)z^j&=&\sum_{j=0}^{\infty}\left(\sum_{i=j+1}^{\infty}\overline{g}_{k}(i)\right)z^j
=\sum_{i=1}^{\infty}\sum_{j=0}^{i-1}\overline{g}_{k}(i) z^j=\frac {1} {1-z}\sum_{i=0}^{\infty} \overline{g}_{k}(i) (1-z^i).\label{Lemma 4.1proof-1}
\end{eqnarray}
Next, we proceed with the mathematical induction on $k$. For $k=0$,
\[
\sum_{j=0}^{\infty}\overline{g}_{1}(j)z^j=\frac {1-G(z)} {1-z}=\widehat{G}_0(z)\quad\mbox{(by (\ref{Lemma 4.1proof-1}), (\ref{G_n(z)}) and (\ref{widehat{G}_n (z)}))}.
\]
Under the induction hypothesis that (\ref{Lemma 4.1formula}) holds for $k=i-1\in \{0,1,\cdots,n-1\}$, we have
\begin{eqnarray}
\sum_{j=0}^{\infty}\overline{g}_{i+1}(j)z^j&=&\frac {\widehat{G}_{i-1}(1)-\widehat{G}_{i-1}(z)} {1-z}\quad\mbox{(by (\ref{Lemma 4.1proof-1}) and the induction hypothesis )} \nonumber\\
&=&\frac {\gamma_{i}/i!-\widehat{G}_{i-1}(z)} {1-z} \quad\mbox{(by (\ref{widehat{G}_{n-1}(1)}))} \nonumber\\
&=&\widehat{G}_i(z)\quad\mbox{(by (\ref{widehat{G}_{n-1}(z)}))}.\nonumber
\end{eqnarray}
Therefore, (\ref{Lemma 4.1formula}) holds for $k=i\in \{1,2,\cdots,n\}$, which completes the proof.
\end{proof}

\bigskip

The following lemma is referred to the Karamata's Tauberian theorem for power series.
\begin{lemma}[Bingham et al.~\cite{Bingham:1989}, p.40]
\label{Lemma 4.2}
 Let $\{q(j)\}_{j=0}^{\infty}$ be a non-negative sequence
such that $Q(z)\stackrel{\rm def}{=}\sum_{j=0}^{\infty}q(j)z^j$ converges for $0\le z<1$,
let $L(\cdot)$ be slowly varying at $\infty$, and $b\ge 0$, then the following two statements are equivalent:
\begin{flalign}
\begin{split}
\mbox{(i)}&\quad Q(z)\sim(1-z)^{-b}L\left(1/(1-z)\right),\quad z\uparrow 1; \label{Lemma 4.2formula-1} \mbox{ and }
\end{split}&\\
\begin{split}
\mbox{(ii)}&\quad \sum_{k=0}^{j}q(k)\sim\frac 1 {\Gamma(b+1)}j^{b}L(j), \quad j\to\infty.\label{Lemma 4.2formula-2}
\end{split}&
\end{flalign}

Furthermore, if the sequence $\{q(j)\}_{j=0}^{\infty}$ is ultimately monotonic and $b>0$, then both (i) and (ii) are equivalent to
\begin{flalign}
\begin{split}
\mbox{(iii)}&\quad q(j)\sim\frac 1 {\Gamma(b)}j^{b-1}L(j),\quad j\to\infty.  \label{Lemma 4.2formula-3}
\end{split}&
\end{flalign}
\end{lemma}

\begin{lemma}\label{Lemma 4.3}
Let $\{g(j)\}_{j=0}^{\infty}$ be a discrete probability distribution with the GF $G(z)$.
Assume that $n< d<n+1$ for some $n\in\{0,1,2,\cdots\}$. The sequence $\{\overline{g}_{n+1}(j)\}_{j=0}^{\infty}$ is defined by (\ref{overline{g}_{n+1}(j)}).
Let $L(\cdot)$ be slowly varying. The following two statements are equivalent:
\begin{flalign}
\begin{split}
\mbox{(i)}&\quad G_n(z)\sim(1-z)^{d}L(1/(1-z)),\quad z\uparrow 1; \label{Lemma 4.3formula-1} \mbox{ and }
\end{split}&\\
\begin{split}
\mbox{(ii)}&\quad \overline{g}_{1}(j)\sim\frac{\Gamma(d)} {\Gamma(d-n)\Gamma(n+1-d)} j^{-d} L(j), \quad j\to\infty.\label{Lemma 4.3formula-2}
\end{split}&
\end{flalign}
\end{lemma}

\begin{proof} By the definition of $\widehat{G}_n(z)$ in (\ref{widehat{G}_n (z)}),
(\ref{Lemma 4.3formula-1}) is equivalent to
\begin{eqnarray}
\widehat{G}_n(z)\sim(1-z)^{-(n+1-d)} L\left(1/(1-z)\right). \label{Lemma 4.3proof-1}
\end{eqnarray}
Note that $0< n+1-d<1$ and
the sequence $\{\overline{g}_{n+1}(j)\}_{j=0}^{\infty}$ is decreasing with the GF $\widehat{G}_n(z)$ (by Lemma \ref{Lemma 4.1}).
Applying Lemma \ref{Lemma 4.2} (taking $b=n+1-d$ in (\ref{Lemma 4.2formula-1}) and (\ref{Lemma 4.2formula-3})), we know that (\ref{Lemma 4.3proof-1}) is equivalent to
\begin{eqnarray}
\overline{g}_{n+1}(j)\sim\frac {1} {\Gamma(n+1-d)} j^{-d+n} L(j),\quad j\to\infty.  \label{Lemma 4.3proof-2}
\end{eqnarray}

Next, we prove the equivalence of (\ref{Lemma 4.3formula-2}) and (\ref{Lemma 4.3proof-2}).
Noting the recursive relation (\ref{overline{g}_{n+1}(j)}) and repeatedly applying Lemma \ref{Lemma 3.3}, (\ref{Lemma 4.3proof-2}) is equivalent to
\begin{eqnarray}
\overline{g}_{1}(j)&\sim& \frac {(d-1)\cdots (d-n)} {\Gamma(n+1-d)} j^{-d} L(j),\quad j\to\infty.
\end{eqnarray}
Note that $\Gamma(d)=(d-1)\cdots (d-n)\Gamma(d-n)$.
The proof is completed.
\end{proof}

\begin{definition}[e.g., Bingham et al.~\cite{Bingham:1989}, or Resnick~\cite{Resnick:2008}] \label{II}
 A function $F:$ $(0,\infty)\to (0,\infty)$ belongs to the de Haan class $\Pi$ at $\infty$ if there exists a function $H:(0,\infty)\to (0,\infty)$ such that
\begin{eqnarray}
\lim_{t\uparrow \infty}\frac {F(xt)-F(t)}{H(t)}&=&\log x\quad\mbox{ for all }x>0,\label{deHaan1}
\end{eqnarray}
where the function $H$ is called the auxiliary function of $F$.
Similarly, $F(t)$ belongs to the class $\Pi$ at $0$ if $F(1/t)$ belongs to $\Pi$ at $\infty$, or equivalently, there exists a function $H:(0,\infty)\to (0,\infty)$ such that
\begin{eqnarray}
\lim_{u\downarrow 0}\frac {F(xu)-F(u)}{H(u)}&=&-\log x\quad\mbox{ for all }x>0.\label{deHaan2}
\end{eqnarray}
\end{definition}

\section{Proof of $\Delta$-analyticity}
\label{app-B}

Recall that a domain is a $\Delta$-domain at $1$ if it is in the form
$\Delta=\{z|\ |z|<R,\ z\neq 1,\ |\mbox{Arg}(z-1)|>\phi\}$ for some $R>1$ and $0<\phi<\frac {\pi} 2$. A function is $\Delta$-analytic at $1$ if it is analytic in some $\Delta$-domain at $1$.

\subsection{Example~A}

For this example, we prove, in the following, that $Ez^{L_{\infty}}$ can be analytically continued to $\mathbb{C}\backslash [1,\infty)$, which implies the $\Delta$-analyticity.

Recall that (\ref{EzL{infty}}) and (\ref{exmp-Kcirc}). Note that $\beta(s)$ and $\beta^{(e)}(s)$ are analytic in $\mathbb{D}=\mathbb{C}\backslash(-\infty,0]$, and both $X(z)$ and $X^{(de)}(z)$ can be analytically continued to
$\mathbb{C}\backslash [1,\infty)$.
For $Ez^{L_{\infty}}$ to be analytic in $\mathbb{C}\backslash [1,\infty)$,  we only need to verify the
two things: (i) $\lambda (1-X(z))\in \mathbb{D}$ for any $z\in \mathbb{C}\backslash [1,\infty)$; (ii)
$1-\rho\beta^{(e)}(\lambda-\lambda X(z))\cdot X^{(de)}(z)$ is nonzero in $\mathbb{C}\backslash [1,\infty)$.

(i) Let $s=(1-z)/(1-q)=x+iy$. By (\ref{exmp-X(z)-asym}),
\[
1-X(z)=\frac {s} {1+qs} = \frac {(x+iy)(1+qx-iqy)} {(1+qx)^2+(qy)^2}=\frac {x(1+qx)+qy^2+iy} {(1+qx)^2+(qy)^2}.
\]
It follows that if $\mbox{Im}(z)\not= 0$ (or say $y\not= 0$), then $\mbox{Im}\big(1-X(z)\big)\not= 0$. Additionally, for any real $z\in (-\infty,1)$ (or say $s>0$), $1-X(z)=s/(1+qs)>0$. Hence, $\lambda (1-X(z))\in \mathbb{D}$ for any $z\in \mathbb{C}\backslash [1,\infty)$.

(ii) By Lemma 1 in \cite{Ramsay:2007}, for $s\in \mathbb{D}$,
\begin{eqnarray}
\beta^{(e)}(s)&=&\frac {a-1} {\Gamma(a)} \int_0^{\infty}\frac {t^{a-1}e^{-t}} {t+b s} dt. \label{exmpA-beta(e)-333}
\end{eqnarray}

Set $s=(1-z)/(1-q)=x+iy$. By the expressions of $1-X(z)=s/(1+qs)$ and $X^{(de)}(z)=1/(1+qs)$, and using (\ref{exmpA-beta(e)-333}), we have
\begin{eqnarray}
\beta^{(e)}(\lambda-\lambda X(z))\cdot X^{(de)}(z)
&=&\beta^{(e)}\Big(\frac {\lambda s} {1+qs}\Big)\cdot \frac {1} {1+qs}\nonumber\\
&=&\frac {a-1} {\Gamma(a)} \int_0^{\infty}\frac {t^{a-1}e^{-t}} {(1+qs)t+\lambda b s} dt\nonumber\\
&=&\frac {a-1} {\Gamma(a)} \int_0^{\infty}\frac {t^{a-1}e^{-t}} {(1+qx)t+\lambda b x+iy(qt+\lambda b)} dt\nonumber\\
&=&\frac {a-1} {\Gamma(a)} \int_0^{\infty}\frac {[(1+qx)t+\lambda b x-iy(qt+\lambda b)]t^{a-1}e^{-t}} {[(1+qx)t+\lambda b x]^2+[y(qt+\lambda b)]^2} dt,
\end{eqnarray}
from which, we know that $\mbox{Im}(\beta^{(e)}\left(\lambda-\lambda X(z))\cdot X^{(de)}(z)\right) \neq 0$ for $y\neq 0$ or $\mbox{Im}(z)\not=0$. So,
$1-\rho\beta^{(e)}(\lambda-\lambda X(z))\cdot X^{(de)}(z)$ is nonzero in $\mathbb{C}\backslash (-\infty,\infty)$.
In addition, for any real $z\in (-\infty,1)$, we have $s=x> 0$, then
$\beta^{(e)}(\lambda-\lambda X(z))\cdot X^{(de)}(z)=\beta^{(e)}\big(\frac {\lambda x} {1+qx}\big)\cdot \frac {1} {1+qx}< 1$. So,
$1-\rho\beta^{(e)}(\lambda-\lambda X(z))\cdot X^{(de)}(z)$ is nonzero for $z\in(-\infty,1)$. This completes the proof that
$Ez^{L_{\infty}}$ is analytic in $\mathbb{C}\backslash [1,\infty)$, which obviously implies the $\Delta$-analyticity of $Ez^{L_{\infty}}$ at 1.

\subsection{Example~B}

Similarly, for this example, we can show that $Ez^{L_{\infty}}$ can be analytically continued to $\mathbb{C}\backslash [1,\infty)$, which implies the $\Delta$-analyticity.

Note that $\tau(s)$ and $\tau^{(e)}(s)$, $\beta(s)$ and $\beta^{(e)}(s)$ can be analytically continued to $\mathbb{D}=\mathbb{C}\backslash(-\infty,0]$, and both $X(z)$ and $X^{(de)}(z)$ can be analytically continued to
$\mathbb{C}\backslash [1,\infty)$.
For $Ez^{L_{\infty}}$ to be analytic in $\mathbb{C}\backslash [1,\infty)$,  we only need to verify the
two things: (i) $\lambda (1-X(z))\in \mathbb{D}$ for any $z\in \mathbb{C}\backslash [1,\infty)$; (ii)
$1-\rho\beta^{(e)}(\lambda-\lambda X(z))\cdot X^{(de)}(z)$ is nonzero in $\mathbb{C}\backslash [1,\infty)$.

(i) Similar to (\ref{exmpA-beta(e)-333}), we know that
\begin{eqnarray}
\tau^{(e)}(s)=\frac {u-1} {\Gamma(u)} \int_0^{\infty}\frac {t^{u-1}e^{-t}} {t+v s} dt, \qquad\mbox{ for }s\in \mathbb{\mathbb{D}}=\mathbb{C}\backslash(-\infty,0].
\end{eqnarray}
By the definition of $X(z)$, we can immediately write
\begin{eqnarray}
1-X(z)=1-zX_0(z)=1-z\tau(1-z).\label{examp2-delta-analy-1}
\end{eqnarray}
Set $s=1-z=x+iy$ in the above equality. Note that $\tau(s)=1-\tau_1 s\tau^{(e)}(s)$. We then have
\begin{eqnarray}
1-X(z)=1-(1-s)\tau(s)&=&s\big[1+\tau_1\cdot (1-s)\tau^{(e)}(s)\big]=s\varphi(s),\label{examp2-delta-analy-2}
\end{eqnarray}
where
\begin{eqnarray}
\varphi(s)&=&1+\tau_1\cdot (1-s)\tau^{(e)}(s)\nonumber\\
&=&1+ \frac {v} {\Gamma(u)} \int_0^{\infty}\frac {1-s} {t+v s}t^{u-1}e^{-t} dt\nonumber\\
&=&1+ \frac {v} {\Gamma(u)} \int_0^{\infty}\frac {(1-x)(t+vx)-vy^2-(t+v)yi} {(t+v x)^2+(vy)^2}t^{u-1}e^{-t} dt\nonumber\\
&\stackrel{\rm def}{=}&f(x,y)-iyg(x,y). \label{examp2-delta-analy-200}
\end{eqnarray}
It is worthwhile to mention that $g(x,y)=\frac {v} {\Gamma(u)} \int_0^{\infty}\frac {t+v} {(t+v x)^2+(vy)^2}\cdot t^{u-1}e^{-t} dt>0$ for all $x>0$ and $y\not=0$.
It follows from (\ref{examp2-delta-analy-2}) and (\ref{examp2-delta-analy-200}) that
\begin{eqnarray}
\mbox{Im}\big(1-X(z))&=&y[f(x,y)-xg(x,y)]\nonumber\\
&=&y\Big[1+\frac {v} {\Gamma(u)} \int_0^{\infty}\frac {t-2xt-vx^2-vy^2} {(t+v x)^2+(vy)^2}\cdot t^{u-1}e^{-t} dt\Big]\nonumber\\
&=&y\cdot\frac {1} {\Gamma(u)} \int_0^{\infty}\frac {vt+t^2} {(t+v x)^2+(vy)^2}\cdot t^{u-1}e^{-t} dt.
\end{eqnarray}
Therefore,
if $\mbox{Im}(z)\not= 0$ (or $y\not= 0$), then $\mbox{Im}\big(1-X(z)\big)\not= 0$. Additionally, for any real $z\in (-\infty,1)$ (or $s>0$), $1-X(z)\ge 1-(1-s)>0$ (see (\ref{examp2-delta-analy-2})). Hence, $\lambda (1-X(z))\in \mathbb{D}$ for any $z\in \mathbb{C}\backslash [1,\infty)$.

(ii) It follows from (\ref{examp2-delta-analy-2}) and (\ref{examp2-delta-analy-200}) that
\begin{eqnarray}
X^{(de)}(z)&=&\frac {1-X(z)}{\chi_1(1-z)}=\frac {\varphi(s)} {1+ \tau_1},\label{examp2-delta-analy-3}\\
\frac 1 {\varphi(s)}&=&\frac {f(x,y)+iyg(x,y)} {(f(x,y))^2+(yg(x,y))^2}\stackrel{\rm def}{=}f_0(x,y)+iyg_0(x,y),\label{examp2-delta-analy-7}
\end{eqnarray}
where $g_0(x,y)=\frac {g(x,y)} {(f(x,y))^2+(yg(x,y))^2}>0$ for all $x>0$ and $y\not=0$.
It follows from (\ref{examp2-delta-analy-2}), (\ref{examp2-delta-analy-3}), and (\ref{exmpA-beta(e)-333}) that
\begin{eqnarray}
&&\beta^{(e)}(\lambda-\lambda X(z))\cdot X^{(de)}(z)=\beta^{(e)}\big(\lambda s\varphi(s)\big)\cdot \frac {\varphi(s)} {1+\tau_1}\nonumber\\
&=&\frac {a-1} {\Gamma(a)(1+\tau_1)} \int_0^{\infty}\frac {t^{a-1}e^{-t}} {\frac t {\varphi(s)}+\lambda b s} dt\nonumber\\
&=&\frac {a-1} {\Gamma(a)(1+\tau_1)} \int_0^{\infty}\frac {t^{a-1}e^{-t}} {tf_0(x,y)+\lambda b x +iy(tg_0(x,y)+\lambda b)} dt\quad\mbox{(by  (\ref{examp2-delta-analy-7}))}\nonumber\\
&=&\frac {a-1} {\Gamma(a)(1+\tau_1)} \int_0^{\infty}\frac {[tf_0(x,y)+\lambda b x -iy(tg_0(x,y)+\lambda b)]t^{a-1}e^{-t}}
{[tf_0(x,y)+\lambda b x]^2+y^2(tg_0(x,y)+\lambda b)^2} dt,
\end{eqnarray}
from which, we know that $\mbox{Im}(\beta^{(e)}\left(\lambda-\lambda X(z))\cdot X^{(de)}(z)\right) \neq 0$ for all $y\neq 0$. So,
$1-\rho\beta^{(e)}(\lambda-\lambda X(z))\cdot X^{(de)}(z)$ is nonzero in $\mathbb{C}\backslash (-\infty,\infty)$.
Furthermore, for any real $z\in (-\infty,1)$, we have $s=x> 0$ and by (\ref{examp2-delta-analy-200}),
\begin{eqnarray}
\varphi(x)&=&1+ \frac {v} {\Gamma(u)} \int_0^{\infty}\frac {1-x} {t+v x}t^{u-1}e^{-t} dt=\frac {1} {\Gamma(u)} \int_0^{\infty}\frac {t+v} {t+v x}t^{u-1}e^{-t} dt>0,\nonumber\\
\varphi(x)&<&1+ \frac {v} {\Gamma(u)} \int_0^{\infty}\frac {1} {t}t^{u-1}e^{-t} dt=1+\frac {v} {u-1}=1+\tau_1,
\end{eqnarray}
hence
$\beta^{(e)}(\lambda-\lambda X(z))\cdot X^{(de)}(z)=\beta^{(e)}\big(\lambda x\varphi(x)\big)\cdot \frac {\varphi(x)} {1+\tau_1}< 1$. So,
$1-\rho\beta^{(e)}(\lambda-\lambda X(z))\cdot X^{(de)}(z)$ is nonzero for $z\in(-\infty,1)$. Now, we have completed the proof that
$Ez^{L_{\infty}}$ is analytic in $\mathbb{D}=\mathbb{C}\backslash [1,\infty)$, which implies the $\Delta$-analyticity of $Ez^{L_{\infty}}$ at 1.

\end{document}